\documentclass[10pt]{article}
 \usepackage[left=1.5cm, right=1.5cm, top=1.5cm]{geometry} 

\usepackage{amsmath}
\usepackage{comment}
\usepackage{amsthm}
\usepackage{amsfonts}
\usepackage{fancyhdr}
\usepackage{amssymb}
\usepackage{amscd}
\usepackage{accents}
\usepackage{tikz}
\usetikzlibrary{decorations.markings}
\usetikzlibrary{chains}
\usepackage{mathtools}
\usepackage[T1]{fontenc}
\usepackage{selinput}
\SelectInputMappings{%
	adieresis={ä},
	eacute={é},
	Lcaron={Ľ},
}

\usepackage[all,cmtip]{xy}
\usepackage{hyperref}
\usepackage{graphicx}
\graphicspath{ {images/} }
\usepackage[nottoc,notlot,notlof]{tocbibind}
\hypersetup{
	colorlinks=true,
	linkcolor=blue,
	filecolor=magenta,      
	urlcolor=cyan,
	citecolor=blue,
}

\numberwithin{equation}{subsection}
\newtheorem{theorem}{Theorem}[section]
\newtheorem{lemma}[theorem]{Lemma}
\newtheorem{corollary}[theorem]{Corollary}

\theoremstyle{plain}
\newtheorem{definition}[theorem]{Definition}

\newtheorem{proposition-definition}[theorem]{Proposition-Definition}
\newtheorem{proposition}[theorem]{Proposition}

\newtheorem{thm}{Theorem}[subsection]

\theoremstyle{remark}
\newtheorem{remark}[theorem]{Remark}

\newcommand{\Lie}{\mathrm{Lie}}

\newcommand{\Z}{\mathbb Z}

\newcommand{\F}{\mathbb F}

\usepackage{authblk}
\begin{document}

\title{Presentation of the Iwasawa algebra of the pro-$p$ Iwahori subgroup of $GL_n(\mathbb{Z}_p)$}
\author[]{Jishnu Ray}

\affil[]{Département de Mathématiques, Bâtiment 307,\\
Faculté des Sciences d'Orsay, Université Paris-Sud XI,
F-91405 Orsay Cedex\\ 
 
 jishnuray1992@gmail.com; jishnu.ray@u-psud.fr}
\date{}

\maketitle

\begin{abstract}
Iwasawa algebras of compact $p$-adic Lie groups are completed group algebras with applications in number theory in studying class numbers of towers of number fields and  representation theory of $p$-adic Lie groups. We previously determined an explicit presentation of the Iwasawa algebra for the first principal congruence kernel of Chevalley groups over $\mathbb{Z}_p$. In this paper, 
for prime $p>n+1$, we extend our result to determine the explicit presentation,  in the form of generators and relations, of the Iwasawa algebra of the pro-$p$ Iwahori subgroup of $GL_n(\mathbb{Z}_p)$.
\end{abstract}

\section{Introduction}\label{intro:prop}

Let  $G$ be the pro-$p$ Iwahori subgroup of $SL_n(\mathbb{Z}_p)$, that is, the group of matrices in $SL_n(\mathbb{Z}_p)$ which are upper unipotent modulo the ideal $p\mathbb{Z}_p$.

The Iwasawa algebra of $G$ is a non-commutative completed group algebra defined by 
$$\Lambda (G):=\varprojlim _{N\in \mathcal{N}(G)}(G/N),$$ where $\mathcal{N}(G)$ is the set of open normal subgroups in $G$. This algebra has many applications in number theory and $p$-adic representation theory (\cite{Emertonbook}). It is used by Iwasawa \cite{Iwa}  to study the growth of class numbers in towers of number fields. Schneider and Teitelbaum use the Iwasawa algebra to study the category of $\mathbb{Q}_p$-Banach representations of compact $p$-adic Lie groups \cite{SchneiderBan}. In \cite{Ray}, we found an explicit presentation of the Iwasawa algebra for the first principal congruence kernel of a semi-simple, simply connected Chevalley group over $\mathbb{Z}_p$.  In this section, our goal is to extend the method to give an explicit presentation of the Iwasawa algebra for the pro-$p$ Iwahori subgroup of $SL(n,\mathbb{Z}_p)$ generalizing the work of Clozel for $n=2$ \cite{Clozel2}. Our main result (see Theorem \ref{eq:theoremgln}) is the following.\\

\textbf{\textit{Theorem.}} \textit{For} $p>n+1$, \textit{the Iwasawa algebra of the pro-}$p$ \textit{Iwahori subgroup of} $SL(n,\mathbb{Z}_p)$ \textit{is naturally isomorphic as a topological ring to} $\mathcal{A}/\mathcal{R}$. \\

Here $\mathcal{A}$ is a non-commutative power series ring over $\Z_p$ in  the variables $V_{\alpha},W_{\delta}, U_{\beta}$ where $\alpha,\delta,\beta$ varies over the positive, simple and negative roots respectively. The ordering among the variables is given by the ordering on the roots as in Theorem \ref{eq:orderedbasisiwahori}. The algebra   $\mathcal{R}$ is a closed two-sided ideal in $\mathcal{A}$ generated by a set of explicit relations between these variables  (\ref{eq:firstIwahori}-\ref{eq:fifteenIwahori}). 

We first deduce the explicit presentation for $\Omega_G:=\Lambda(G) \otimes_{\mathbb{Z}_p}\mathbb{F}_p$, the Iwasawa algebra modulo $p$ (theorem \ref{eq:secondimptheoremIwahori}) and then we lift the coefficients to $\mathbb{Z}_p$.
We also extend our proof to the case when $G$ is the pro-$p$ Iwahori of $GL(n,\mathbb{Z}_p)$ (corollary \ref{eq:theoremgln}).\\

 The explicit presentation can be used to define a "formal base change map" \cite{Clozel2} of Iwasawa algebras $$\Lambda_L \rightarrow \Lambda_{\mathbb{Q}_p}$$ 
where $\Lambda_L$ is the Iwasawa algebra over a finite unramified extension $L$ of $\mathbb{Q}_{p}$. Such a formal base change map is given by power series which only converge for the globally analytic distributions (continuous dual of the rigid-analytic functions) on the pro-$p$ Iwahori seen as a rigid-analytic space (\textit{loc.cit}). This leads us also to the study of the globally analytic vectors of $p$-adic representations which will be our object in a forthcoming paper. 

Apart from the above implications of our explicit presentation of the Iwasawa algebra, Dong Han and Feng Wei, quoted an earlier version of this paper and noted that our results may provide possible ways to answer the open question on the existence of non-trivial normal elements in $\Omega_G:=\Lambda(G) \otimes_{\Z_p} \F_p$ (cf. introduction and section 5 of \cite{HW1}). An element $r \in \Omega_G$ is normal if $r\Omega_G=\Omega_Gr$. The question on the normal elements was originally posed in \cite{WB}, later reformulated in \cite{HW1} having dealt with the case for $SL(2,\mathbb{Z}_p)$ and $SL(3,\mathbb{Z}_p)$. As noted in \cite{HW1}, the normal elements help in constructing reflexive ideals in the Iwasawa algebra. The main question of Han and Wei is to find a mechanism for constructing ideals of completed group algebras without using central elements or closed normal subgroups which provide natural ways to construct ideals in the Iwasawa algebra (\textit{loc.cit}).\\

\textbf{Roadmap}. Section \ref{sub:notationiwahori} is devoted to preliminary notations. The theorem on ordered basis of the pro-$p$ Iwahori group $G$ with respect to the $p$-valuation $\omega$ (\ref{defiw}) on $G$ is given in section $\ref{sub:ordiwa}$ (theorem \ref{eq:orderedbasisiwahori}).  In section $\ref{sub:Iwasawa}$ we recall the notion of the Iwasawa algebra $\Lambda(G)$ of $G$, the algebra of $p$-adic measures on $G$. By a theorem of Lazard, this algebra is isomorphic as a $\mathbb{Z}_p$-module to a ring of power series in several variables. The relations between the variables inside $\Lambda(G)$ are given in lemma \ref{relations} in section $\ref{sub:relationsinIwahori}$. The computations for the proof of lemma \ref{relations} are included in section \ref{appendixCrelations}. Sections $\ref{sub:Iwasawa}$ and $\ref{sub:conti}$ deal with the construction of a continuous surjection $\varphi:\mathcal{A} \rightarrow \Lambda(G)$. Let $\overline{\mathcal{B}}:=\overline{\mathcal{A}}/\overline{\mathcal{R}}$ be the reduction modulo $p$ of $\mathcal{A}/\mathcal{R}$. The algebra  $\overline{\mathcal{B}}$ has a natural grading discussed after Theorem $\ref{eq:bird}$. For integer $m \geq 0$, we provide an upper bound on the dimension of the $m$-th graded piece gr$^m\overline{\mathcal{B}}$ in section
\ref{sub:longlemmaiwahori} (see lemma \ref{eq:longlemma}). The computations for the proof of lemma \ref{eq:longlemma} is included in section \ref{appendixDlong}. Finally, the proof of the explicit presentation of $\Omega_G$ and $\Lambda(G)$ is  in section $\ref{sub:maintheoremsiwahori}$.\\

\section{Lazard's ordered Basis for the pro-$p$ Iwahori subgroup $G$}\label{ordIwa}
Recall that $G$ is the pro-$p$ Iwahori subgroup of $SL_n(\mathbb{Z}_p)$, i.e. $G$ is the group of matrices in $SL_n(\mathbb{Z}_p)$ which are upper unipotent modulo the maximal ideal $p\mathbb{Z}_p$ of $\mathbb{Z}_p$. The goal of this section is to  find an ordered basis in the sense of Lazard \cite{Lazard} for the pro-$p$ Iwahori group $G$. This will later be used in deducing an explicit presentation of the Iwasawa algebra of $G$ (theorem \ref{eq:maintheoremIwahori}).

In this section, we first define  a function $\omega_{\frac{1}{n}}$ on $G$ which is a $p$-valuation $(\ref{eq:valuationiwahori})$ in the sense of Lazard (cf. chapter III, $2.1.2$ of \cite{Lazard}). Then, in section $\ref{sub:ordiwa}$,  we give an ordered basis for the $p$-valuation (Theorem $\ref{eq:orderedbasisiwahori}$). This means that we find an ordered set of elements $g_1,...,g_d \in G$ such that 
\begin{flushleft}
	1.  The map $\mathbb{Z}_p^d \rightarrow G$ sending 
	$(z_1,...,z_d) \mapsto g_1^{z_1} \cdots g_d^{z_d}$ is a bijection and\\
	2. $\omega_{\frac{1}{n}}(g_1^{z_1} \cdots g_d^{z_d})=\min_{1\leq i\leq d}(\omega_{\frac{1}{n}}(g_i)+val_p(z_i))$.
\end{flushleft}

\subsection{$p$-valuation on $G$}\label{sub:notationiwahori}
Let $p$ be a prime number. Fix a pinning \cite[XXIII 1]{SGA3.3r} of the split reductive group $SL_n$ over $\mathbb{Z}_p$
\[
\left(T_S,M,\Phi,\Pi,(X_{\varsigma})_{\varsigma\in\Pi}\right)
\]
where $T_S$ is a split maximal torus in $SL_n$, $M=X^{\ast}(T_S)$ is its group of characters,
\[
\mathfrak{g}=\mathfrak{g}_{0}\oplus\oplus_{\varsigma \in \Phi}\mathfrak{g}_{\varsigma}
\]
is the weight decomposition for the adjoint action of $T_S$
on $\mathfrak{g}=\Lie(SL_n)$, $\Pi\subset \Phi$ is a basis
of the root system $\Phi \subset M$ and for each $\varsigma \in \Pi$,
$X_{\varsigma}$ is a $\mathbb{Z}_{p}$-basis of $\mathfrak{g}_{\varsigma}$. \\

We denote the positive and negative roots by $\Phi^+$ and $\Phi^-$ respectively with respect to the standard Borel subgroup of $SL_n(\mathbb{Z}_p)$. The height function on the roots $h(\varsigma) \in \mathbb{Z}$ of $\varsigma \in \Phi$ is the sum of the coefficients of $\varsigma$ in the basis $\Pi$ of $\Phi$. We expand
$(X_{\varsigma})_{\varsigma\in\Pi}$ to a Chevalley system $(X_{\varsigma})_{\varsigma\in \Phi}$
of $SL_n$ \cite[XXIII 6.2]{SGA3.3r}. For $\varsigma \in \Phi$, $t \in \mathbb{Z}_p$, $\lambda \in \mathbb{Z}_p^{\ast}$, we denote $x_{\varsigma}(t)=\exp(tX_{\varsigma})$, $h_{\varsigma}(\lambda)=w_{\varsigma}(\lambda)w_{\varsigma}(1)^{-1}$, where $w_{\varsigma}(\lambda)=x_{\varsigma}(\lambda)x_{-\varsigma}(-\lambda^{-1})x_{\varsigma}(\lambda)$.\\

Henceforth, we assume 
\begin{align}
\label{eq:conditioniwahoi}
p>n+1.
\end{align}
Thus, $G$ is $p$-saturated (cf. Lazard, \cite{Lazard}, $3.2.7.5$, chap. $3$ ).

The following construction can be found in \cite[p. 172]{Lie}.
\begin{definition}\label{defiw}
	For each real $a$ with $\frac{1}{p-1}<a<\frac{p-2}{(p-1)(n-1)}$, $g \in G, g =(a_{ij})$,  a $p$-valuation $\omega_{a}$ on $G$ is given by
	\begin{align}\label{eq:valuationiwahori}
	\omega_{a}(g):&=\min\Big(\min_{1 \leq i \neq j \leq n}((j-i)a+val_p(a_{ij})),\min_{1 \leq i \leq n}val_p(a_{ii}-1)\Big),\\
	&=\min \Big({(j-i)a+val_p(a_{ij})}_{i \neq j},val_p(a_{ii}-1)\Big).
	\end{align}
\end{definition}

Setting $a=\frac{1}{n}$, we write simply $\omega$ for $\omega_a$.
The function $\omega$ makes $G$ a $p$-valuable group in the sense of \cite[III.2.1.2]{Lazard}. In fact, for $p>n+1$, $(G,\omega)$ is $p$-saturated \cite[III.3.2.7.5]{Lazard}.\\

The pro-$p$ Iwahori group $G$ has a triangular decomposition 
\begin{center}
	$G=N^-TN^+$
\end{center}
(cf p. $317$, section $1.6$ of \cite{Garrett}, also \cite{Bruhat}), where $N^-$ (resp. $N^+$) is the subgroup of lower (resp. upper) unipotent matrices of $G$ and $T$ is the subgroup of the diagonal matrices of $G$. Let $\alpha,\beta, \delta $ be the roots in $\Phi^+,\Phi^-,\Pi$ respectively. From $\ref{eq:valuationiwahori}$ we deduce that
\begin{align}
\omega(x_{\beta}(p))&=\frac{(j-i)}{n} + 1; \  \beta\in \Phi^-,\beta=(i,j),i>j,\label{2.2.1}\\
\omega(h_{\delta}(1+p))&=1; \ \delta \in \Pi,\label{2.2.2}\\
\omega(x_{\alpha}(1))&=\frac{(j-i)}{n}; \ \alpha \in \Phi^+,\alpha=(i,j),i<j.\label{2.2.3}
\end{align}
In the next section we are going to use this $p$-valuation $\omega$ on $G$ in order to find an ordered basis of $G$ in the sense of Lazard (cf. Lazard, \cite[III, 2.2]{Lazard}).
\subsection{Ordered basis}\label{sub:ordiwa}

In this section we find an ordered basis for the $p$-valuation on $G$ (theorem \ref{eq:orderedbasisiwahori}). But before proving theorem \ref{eq:orderedbasisiwahori} we need  a preparatory lemma \ref{lemma1} and a proposition \ref{matrixXYZexplicit}. Let $E_{i,j}$ be the standard elementary matrix at $(i,j)$-th place.

	\begin{lemma}\label{lemma1}
		Any element $g \in G$ has a unique expression of the form
		\begin{equation}
		g=\prod_{{\beta} \in \Phi^-}x_{\beta}(u_{\beta})\prod_{{\delta} \in \Pi}h_{\delta}(1+v_{\delta})\prod_{{\alpha} \in \Phi^+}x_{\alpha}(w_{\alpha}),
		\end{equation}
		where $u_{\beta},v_{\delta} \in p\mathbb{Z}_p,w_{\alpha} \in \mathbb{Z}_p$. The order of the products is taken as follows (compare also theorem \ref{eq:orderedbasisiwahori}):
		\begin{flushleft}
			(i) first take the lower unipotent matrices in the order given by the (increasing) height function on the roots,\\
			(ii) then take the diagonal elements $h_{\delta}(1+p)$ for $\delta \in \Pi$ starting from the top left extreme to the low right extreme and, 
			
			(iii) finally, take the upper unipotent matrices in the following lexicographic order:\\
			The matrix $(1+E_{i,j})$ comes before $(1+E_{k,l})$ if and only if $i\geq k$ and $i=k$ $\implies$ $j>l$. 
		\end{flushleft}
	\end{lemma}
		
		That is, for the upper unipotent matrices we start with the low and right extreme and then fill the lines from the right, going up.  \\
		
\begin{proof}
	\vspace{.2cm}
$\bullet$ 
Any element of $T$ can be uniquely written as $\prod_{\delta \in \Pi}h_{\delta}(1+v_{\delta})$ with $v_{\delta} \in p\mathbb{Z}_p$. (cf. last paragraph of the proof of Theorem $3.2$ of \cite{Ray}.) If $\delta=(i,i+1)$, then $h_{\delta}(1+v_{\delta})$ is the diagonal matrix $(1+p)^{v_{\delta}}E_{i,i}+(1+p)^{-v_{\delta}}E_{i+1,i+1}+\sum_{j \neq i}E_{j,j}$.  So, each element $Z$  of $T$ can be uniquely written as 
	\begin{align}
	\label{eq:matrixZ}
	Z=\sum_{k=1}^n(1+p)^{x_{k,k}-x_{k-1,k-1}}E_{k,k}
	\end{align}
	for unique $x_{k,k} \in \mathbb{Z}_p$ and $x_{0,0}=x_{n,n}=0$.\\
	\hrule
		\vspace{.2cm}
	
$\bullet$	For  $x_{i,j} \in \mathbb{Z}_p$, we have
	\begin{center}
		$(1+px_{n,1}E_{n,1})(1+px_{n-1,1}E_{n-1,1})(1+px_{n,2}E_{n,2})\cdots (1+px_{n,n-1})E_{n,n-1}$\\
		$=1+px_{n,1}E_{n,1}+px_{n-1,1}E_{n-1,1}+px_{n,2}E_{n,2}+\cdots +px_{n,n-1}E_{n,n-1}$,
	\end{center}
	where the order of the product is taken according to the (increasing) height function on the roots. 
	This directly implies that every element $X$ of $N^-$ can be  written as 
	\begin{align}\label{eq:matrixX}
	X=(1+pE_{n,1})^{x_{n,1}}(1+pE_{n-1,1})^{x_{n-1,1}}(1+pE_{n,2})^{x_{n,2}}\cdots (1+pE_{n,n-1})^{x_{n,n-1}}=I+\sum\limits_{\substack{i\in [2,n] \\ j\in [1,i-1]}}px_{i,j}E_{i,j},
	\end{align} 
	with unique $x_{i,j} \in \mathbb{Z}_p$. \\
	
	Thus, each element of $N^-$ can be uniquely written as a product $\prod_{\beta \in \Phi^-}x_{\beta}(p)$ with the ordering given by the height function. \\
\hrule
\vspace{.2cm}

$\bullet$	
	Also, because of our dual lexicographic order on $N^+$, we have
	\begin{flushleft}
		$(1+x_{n-1,n}E_{n-1,n})(1+x_{n-2,n}E_{n-2,n})(1+x_{n-2,n-1}E_{n-2,n-1})\cdots (1+x_{1,2}E_{1,2})$\\
		$=1+x_{n-1,n}E_{n-1,n}+x_{n-2,n}E_{n-2,n}+x_{n-2,n-1}E_{n-2,n-1}+\cdots +x_{1,2}E_{1,2}$.
	\end{flushleft}
	Indeed, for any $b\in \mathbb{N}$, if we take the product of the first $b$ terms in the L.H.S of the above equation, the set of column entries which appears in the subscript of the elementary matrices of the product, is disjoint from the row entry which appears in the subscript of the elementary matrix occurring in the $(b+1)^{th}$ term of the product in the L.H.S. We also use $E_{i,j}E_{k,l}=E_{i,l}$ for $j=k$ and $E_{i,j}E_{k,l}=0$ if $j \neq k$.\\
	
	So, each element $Y$ of $N^+$ can be uniquely written as 
	\begin{align}\label{eq:matrixY}
	Y=(1+E_{n-1,n})^{x_{n-1,n}}(1+E_{n-2,n})^{x_{n-2,n}}(1+E_{n-2,n-1})^{x_{n-2,n-1}}... (1+E_{1,2})^{x_{1,2}}=I+\sum\limits_{\substack{i\in [1,n-1] \\ j\in [i+1,n]}}x_{i,j}E_{i,,j},
	\end{align}
	with $x_{i,j} \in \mathbb{Z}_p$. \\
	\hrule
		\vspace{.2cm}
		
	Therefore, from the triangular decomposition $G=N^-TN^+$, we get that each element $g \in G$ has a unique expression of the form
	\begin{equation}
	g=\prod_{{\beta} \in \Phi^-}x_{\beta}(u_{\beta})\prod_{{\delta} \in \Pi}h_{\delta}(1+v_{\delta})\prod_{{\alpha} \in \Phi^+}x_{\alpha}(w_{\alpha}),
	\end{equation}
	where $u_{\beta},v_{\delta} \in p\mathbb{Z}_p,w_{\alpha} \in \mathbb{Z}_p$. The order of the products is taken according to Theorem $\ref{eq:orderedbasisiwahori}$.  This proves lemma \ref{lemma1}.
	\end{proof}

	In order to find an ordered basis for the $p$-valuation on $G$, we need to compute the  $p$-valuation of the product $XZY$ (the matrices $X,Y,Z$ are defined in the proof of lemma \ref{lemma1}). 
	\begin{proposition}\label{matrixXYZexplicit}
		Let $a_{i,j}$ be the $(i,j)^{th}$ entry of the matrix $XZY=(a_{i,j})$. Then,
		\begin{align}
		a_{1,j}&=(1+p)^{x_{1,1}}x_{1,j}, \ j \in [2,n] \label{eq:matrixa1j}\\
		a_{i,1}&=px_{i,1}(1+p)^{x_{1,1}}, \ i \in [2,n]\label{eq:matrixai1}\\
		a_{1,1}&=(1+p)^{x_{1,1}} \label{eq:matrixa11}\\
		a_{i,j}&=(\sum_{k=1}^mpx_{i,k}x_{k,j}(1+p)^{x_{k,k}-x_{k-1,k-1}})+px_{i,j}(1+p)^{x_{j,j}-x_{j-1,j-1}}, \  2 \leq j < i \leq n \label{eq:matrixaijilargej}\\
		a_{i,j}&=(\sum_{k=1}^mpx_{i,k}x_{k,j}(1+p)^{x_{k,k}-x_{k-1,k-1}})+x_{i,j}(1+p)^{x_{i,i}-x_{i-1,i-1}}, \  2 \leq i < j \leq n \label{eq:matrixaijismallj}\\
		a_{i,j}&=(\sum_{k=1}^m px_{i,k}x_{k,j}(1+p)^{x_{k,k}-x_{k-1,k-1}})+(1+p)^{x_{i,i}-x_{i-1,i-1}}, \ 2 \leq i=j \leq n. \label{eq:matrixaijiequalj}
		\end{align}
	\end{proposition}
	\begin{proof}
	Multiplying the matrix $X$ (see equation $\ref{eq:matrixX}$) with $Z$ (see equation $\ref{eq:matrixZ}$) we get
	\begin{align}
	\label{eq:matrixXZ}
	XZ=\sum_{k=1}^n(1+p)^{x_{k,k}-x_{k-1,k-1}}E_{k,k}+\sum\limits_{\substack{i=2,...,n \\ j=1,...,i-1}}px_{i,j}(1+p)^{x_{j,j}-x_{j-1,j-1}}E_{i,j}.
	\end{align}
	So, if we write the $(i,j)^{th}$ coordinate of the matrix $XZ$ by $(XZ)_{i,j}$, then we deduce (from $\ref{eq:matrixXZ}$) that 
	\begin{equation}
	\label{eq:matrixXZijentry}
	(XZ)_{i,1}=px_{i,1}(1+p)^{x_{1,1}},\ i \in [2,n],
	\end{equation}
	and
	\begin{equation}
	\label{eq:matrixXZ11entry}
	(XZ)_{1,1}=(1+p)^{x_{1,1}}, (XZ)_{1,j}=0,  \  j \in [2,n].
	\end{equation}
	Also, for $i \geq 2$, it is clear from equation $\ref{eq:matrixXZ}$ that the $i^{th}$ row of the matrix $XZ$ is 
	\begin{align}\label{eq:rowith}
	[px_{i,1}(1+p)^{x_{1,1}},...,\underbrace{px_{i,i-1}(1+p)^{x_{i-1,i-1}-x_{i-2,i-2}}}_{(XZ)_{i,i-1}},\underbrace{(1+p)^{x_{i,i}-x_{i-1,i-1}}}_{(XZ)_{i,i}},0,0,...,0]
	\end{align}
	In order to compute the $(i,j)^{th}$ entry of the matrix $XZY=(a_{i,j})$ we have to multiply the $i^{th}$ row of the matrix $XZ$ given above by the $j^{th}$ column of the matrix $Y$ (the matrix $Y$ is given in $\ref{eq:matrixY}$). \\
	
	Now, as the first row of the matrix $Y$ is $[1,x_{1,2},x_{1,3},...,x_{1,n}]$, from equation $\ref{eq:matrixXZ11entry}$ it is clear that 
	\begin{align}
	a_{1,j}=(1+p)^{x_{1,1}}x_{1,j}, \ j \in [2,n].
	\end{align}
	As the first column of the matrix $Y$ is $[1,0,...,0]^t$ (here `$t$' denotes the transpose of the row vector), we deduce from equation $\ref{eq:matrixXZijentry}$ and $\ref{eq:matrixXZ11entry}$ that 
	\begin{align}
	a_{i,1}=px_{i,1}(1+p)^{x_{1,1}}, \ i \in [2,n],
	\end{align}
	and 
	\begin{align}
	a_{1,1}=(1+p)^{x_{1,1}}.
	\end{align}
	For $j \geq 2$, the $j^{th}$ column of the matrix $Y$ is 
	\begin{align*}
	[x_{1,j},x_{2,j},...,x_{j-1,j},1,0,...,0]^t.
	\end{align*}
	Therefore, multiplying the $i^{th}$ row of the matrix $XZ$ given by $\ref{eq:rowith}$ with the $j^{th}$ 
	column of the matrix $Y$ given above, for $i, j \geq 2, m=\min \{i-1,j-1\},$ we get the following three subcases \\
	( viz. $i>j,i<j$ and $i=j$ ):
	\begin{align*}
	a_{i,j}&=(\sum_{k=1}^mpx_{i,k}x_{k,j}(1+p)^{x_{k,k}-x_{k-1,k-1}})+px_{i,j}(1+p)^{x_{j,j}-x_{j-1,j-1}}, \  2 \leq j < i \leq n ,\\
	a_{i,j}&=(\sum_{k=1}^mpx_{i,k}x_{k,j}(1+p)^{x_{k,k}-x_{k-1,k-1}})+x_{i,j}(1+p)^{x_{i,i}-x_{i-1,i-1}}, \  2 \leq i < j \leq n ,\\
	a_{i,j}&=(\sum_{k=1}^m px_{i,k}x_{k,j}(1+p)^{x_{k,k}-x_{k-1,k-1}})+(1+p)^{x_{i,i}-x_{i-1,i-1}}, \ 2 \leq i=j \leq n.\\
	\end{align*}
	This completes the proof of proposition \ref{matrixXYZexplicit}.	
	\end{proof}
	
	Proposition \ref{matrixXYZexplicit} gives us the following corollary.
	
	\begin{corollary}\label{corval}
			The valuations of the terms $a_{i,j}$ obtained in equations ($\ref{eq:matrixa1j}-\ref{eq:matrixaijiequalj}$)  are the following:
			\begin{align}
			val_p(a_{1,j})&=val_p(x_{1,j}),\ j \geq 2, \label{2.2.18}\\
			val_p(a_{i,1})&=val_p(px_{i,1}), \ i \geq 2,\label{2.2.19}\\
			val_p(a_{1,1}-1)&=1+val_p(x_{1,1}),\label{2.2.20}\\
			val_p(a_{i,j})&=val_p\Big((\sum_{k=1}^m px_{i,k}x_{k,j})+px_{i,j}\Big),\ i,j \geq 2,i>j, m=j-1,\label{2.2.21}\\
			val_p(a_{i,j})&=val_p\Big((\sum_{k=1}^m px_{i,k}x_{k,j})+x_{i,j}\Big), \ i,j \geq 2,i<j, m=i-1\label{2.2.22}\\
			val_p(a_{i,i}-1)&=val_p\Big((\sum_{k=1}^m px_{i,k}x_{k,i})+p(x_{i,i}-x_{i-1,i-1})\Big), \ i \geq 2,m=i-1.\label{2.2.23}
			\end{align}
	\end{corollary}
	
The following theorem gives an ordered basis  for the $p$-valuation $\omega$ on $G$.
	\begin{theorem}\label{eq:orderedbasisiwahori}
		The elements \
		\begin{center}
			$\{x_{\beta}(p),h_{\delta}(1+p),x_{\alpha}(1);\beta \in \Phi^-,\delta \in \Pi, \alpha \in \Phi^+\}$
		\end{center} 
		form an ordered basis for the $p$-valuation $\omega$ on $G$, where the ordering is as follows:
		\begin{flushleft}
			(i) first take the lower unipotent matrices in the order given by the (increasing) height function on the roots,\\
			(ii) then take the diagonal elements $h_{\delta}(1+p)$ for $\delta \in \Pi$ starting from the top left extreme to the low right extreme and, 
			
			(iii) finally, take the upper unipotent matrices in the following lexicographic order:\\
			The matrix $(1+E_{i,j})$ comes before $(1+E_{k,l})$ if and only if $i\geq k$ and $i=k$ $\implies$ $j>l$. 
		\end{flushleft}
		
	\end{theorem}	
	
	Point (iii) means that for the upper unipotent matrices we start with the low and right extreme and then fill the lines from the right, going up.  \\
\begin{proof}	
Let $g_1,...,g_d$ ($d=|\Phi|+|\Pi| $) denote the ordered basis as in the statement of Theorem $\ref{eq:orderedbasisiwahori}$. Then, by lemma \ref{lemma1}, we have a bijective map
\begin{align*}
\mathbb{Z}_p^d &  \rightarrow G\\
(x_{n,1},...,x_{1,2}) & \rightarrow g_1^{x_{n,1}}\cdots g_d^{x_{1,2}}.
\end{align*}

	In the following, our objective is to show 
	\begin{flushleft}
		$\omega(g_1^{x_{n,1}}\cdots g_d^{x_{1,2}})=\omega(XZY)$\\
		$=\min \Big(\frac{j-i}{n}+val_p(a_{i,j})_{i \neq j},val_p(a_{i,i}-1)\Big),$\\
		$=\min\limits_{\substack{(i_1,j_1)=\beta \in \Phi^-,\ (i_2,j_2)=\alpha \in \Phi^+\\ (t,t+1)=\delta \in \Pi, t \in [1,n-1]}}\Big\{\omega(x_{\beta}(p))+val_p(x_{i_1,j_1}),\omega(x_{\alpha}(1))+val_p(x_{i_2,j_2}),\omega(h_{\delta}(1+p))+val_p(x_{t,t})\Big\}$.
	\end{flushleft}

	Let us define, for $1 \leq i,j \leq n$,
	\begin{align}\label{eq:valuationmatrixiwahori}
	v_{i,j}&=\frac{j-i}{n}+val_p(a_{i,j});(i \neq j),\\
	v_{i,i}&=val_p(a_{i,i}-1).
	\end{align}
	Then we have to show that 
	
	\begin{align}\label{eq:omegaxzy}
	&\omega(g_1^{x_{n,1}}\cdots g_d^{x_{1,2}})=\omega(XZY)=\min_{1 \leq i,j \leq n}(v_{i,j})\\
	&=\min\limits_{\substack{(i_1,j_1)=\beta \in \Phi^-,\ (i_2,j_2)=\alpha \in \Phi^+\\ (t,t+1)=\delta \in \Pi}}\Big\{\omega(x_{\beta}(p))+val_p(x_{i_1,j_1}),\omega(x_{\alpha}(1))+val_p(x_{i_2,j_2}),\omega(h_{\delta}(1+p))+val_p(x_{t,t})\Big\},
	\end{align}
	for $i_1>j_1,i_2<j_2,t=1,...,n-1$.
	The first two equalities of the above equation are obvious (by definition). To prove the last equality we will rearrange the $v_{i,j}$'s and then use induction. First we order the $v_{i,j}$'s, appearing in equation $\ref{eq:omegaxzy}$ in such a way that the indices are given by:

	\begin{equation}\label{eq:orderequation}
	(\text{for all} \ 1 \leq i,j,i^{'},j^{'} \leq n)\ \text{if}  \ \min\{i^{'},j^{'}\}<\min\{i,j\} \ \text{then} \ v_{i^{'},j^{'}} \ \text{comes before} \ v_{i,j}.
	\end{equation} 
	Such an ordering of the $v_{i,j}$'s can be achieved by first taking the $v_{i,j}$'s in the first row and the first column starting from the top left extreme ($v_{1,1},..,v_{1,n},v_{2,1},...,v_{n,1}$), then the second row and the second column ($v_{2,2},...,v_{2,n},v_{3,2},...,v_{n,2}$) and so on. \\
	
	To compute $\omega(g_1^{x_{n,1}}\cdots g_d^{x_{1,2}})=\min(v_{i,j})$ we use induction: As the basic step of the induction process (the zero-th step) we compute $S_0:=\min(v_{1,1},...,v_{1,n},v_{2,1},...,v_{n,1})=\min_{2 \leq i,j \leq n}(v_{1,1},v_{1,j},v_{i,1})$ and then we proceed in stages, adding one $v_{i,j}$ at each stage of induction according to the prescribed order of the $v_{i,j}$, i.e. in the first stage we compute $S_1:=\min\{S_0,v_{2,2}\}$, then in the second stage we compute $S_2:=\min(S_1,v_{2,3})$ and so on until we have completed computing minimum of all the $v_{i,j}$'s. Note that in the last stage, i.e. at the stage $n^2-n-(n-1)$, $S_{n^2-2n+1}=\min_{1 \leq i,j \leq n}\{v_{i,j}\}=\omega(g_1^{x_{n,1}}\cdots g_d^{x_{1,2}})$ (cf. $(\ref{eq:omegaxzy})$ ). \\
	
	From the definition of $v_{i,j}$ (see equation $\ref{eq:valuationmatrixiwahori}$)  and equations ($\ref{2.2.18}-\ref{2.2.20}$) of Corollary \ref{corval} we get, for $2 \leq i,j \leq n$, that
	\begin{flushleft}
		$\min\limits_{\substack{2 \leq i,j \leq n}}(v_{1,1},v_{1,j},v_{i,1})=\min\limits_{\substack{2 \leq i,j \leq n}}\Big\{ \frac{j-1}{n}+val_p(a_{1,j}),\frac{1-i}{n}+val_p(a_{i,1}),val_p(a_{1,1}-1)\Big\}$\\
		$=\min\limits_{\substack{2 \leq i,j \leq n}}\Big\{ \frac{j-1}{n}+val_p(x_{1,j}),\frac{1-i}{n}+1+val_p(x_{i,1}),1+val_p(x_{1,1})\Big\}$\\
		$=\min\limits_{\substack{\alpha=(1,j),\beta=(i,1) \\ \delta=(1,2), 2 \leq i,j \leq n}}\Big\{\omega(x_{\alpha}(1))+val_p(x_{1,j}),\omega(x_{\beta}(p))+val_p(x_{i,1}),\omega(h_{\delta}(1+p))+val_p(x_{1,1})\Big\}$.
	\end{flushleft}
	
	Then, at each stage of the induction, say $q^{th}$ stage, $q \in \mathbb{N}$, of computing the minimum, we compute $S_q=\min(S_{q-1},v_{i,j})$ for some $(i,j)$ with $i,j \geq 2$, where $S_{q-1}$ is the minimum computed in the $(q-1)^{th}$ stage [the subscript $q$ depends on $(i,j)$]. \\
	Henceforth, we fix the coordinate $(i,j)$ appearing in the definition of $S_q$. Note that $S_{q-1}$ is by definition minimum of all the $v_{i^{'},j^{'}}$ appearing before $v_{i,j}$ in the ordering and by induction hypothesis we can assume that 
	
	\begin{equation}\label{eq:Sq-1}
	S_{q-1}=\min_{H} \Big\{\omega(x_{\alpha}(1))+val_p(x_{i_2^{'},j_2^{'}}),\omega(x_{\beta}(p))+val_p(x_{i_1^{'},j_1^{'}}),\omega(h_\delta(1+p))+val_p(x_{t,t})\Big\},
	\end{equation}
	where $H=\Big\{\alpha=(i_2^{'},j_2^{'}) \in \Phi^+, \beta=(i_1^{'},j_1^{'}) \in \Phi^-, \delta=(t,t+1) \in \Pi 
	\ \text{such}  \ \text{that} \  
	v_{i_1^{'},j_1^{'}}<v_{i,j},v_{i_2^{'},j_2^{'}}<v_{i,j},v_{t,t}<v_{i,j}\Big\}$. ( Here, the symbol $<$ denotes the order function and not the ordinary less than symbol, that is, when we write $v_{i_1^{'},j_1^{'}}<v_{i,j}$ we mean that $v_{i_1^{'},j_1^{'}}$ comes before $v_{i,j}$ in the order given by equation $\ref{eq:orderequation}$). By proposition \ref{hadudu} below, we have
	
	\begin{equation}\label{eq:Shq}
		S_q:=\min\{S_{q-1},v_{i,j}\}=\min \{S_{q-1},V_{i,j}\},
		\end{equation}
		where 
		\begin{align*}
		V_{i,j}&=\omega(x_{\beta_2=(i,j)}(p))+val_p(x_{i,j});(\text{if} \ i>j)\\
		V_{i,j}&=\omega(x_{\alpha_2=(i,j)}(1))+val_p(x_{i,j});(\text{if} \ i<j)\\
		V_{i,j}&=\omega(h_{\delta=(i,i+1)}(1+p))+val_p(x_{i,i}-x_{i-1,i-1});(\text{if} \ i=j).
		\end{align*}
		
	Assume \ref{eq:Shq} for now (for the proof see proposition \ref{hadudu} below).

	Therefore, in the last stage 
	(i.e when $q=n^2-n-(n-1)$) of our induction process we will obtain 
	\begin{flushleft}
		$\omega(g_1^{x_{n,1}}\cdots g_d^{x_{1,2}})$\\
		$=\min \Big(\frac{j-i}{n}+val_p(a_{i,j})_{i \neq j},val_p(a_{i,i}-1)\Big),$\\
		$=\min_{1 \leq i,j \leq n}(v_{i,j})$ (with ordered $v_{i,j}$),\\
		$=S_{q=n^2-2n+1}$,\\
		$=\min\limits_{\substack{(i_1,j_1)=\beta \in \Phi^-,\ (i_2,j_2)=\alpha \in \Phi^+\\ (t,t+1)=\delta \in \Pi, t \in [1,n-1]}}\Big\{\omega(x_{\beta}(p))+val_p(x_{i_1,j_1}),\omega(x_{\alpha}(1))+val_p(x_{i_2,j_2}),\omega(h_{\delta}(1+p))+val_p(x_{t,t})\Big\}$.
	\end{flushleft}
	This complete our proof of Theorem $\ref{eq:orderedbasisiwahori}$. 	
	\end{proof}
		
		\begin{proposition}\label{hadudu}
			With notations as in the proof of Theorem \ref{eq:orderedbasisiwahori}, we have \begin{equation}\label{eq:Sq}
			S_q:=\min\{S_{q-1},v_{i,j}\}=\min \{S_{q-1},V_{i,j}\},
			\end{equation}
			where 
			\begin{align*}
			V_{i,j}&=\omega(x_{\beta_2=(i,j)}(p))+val_p(x_{i,j});(\text{if} \ i>j)\\
			V_{i,j}&=\omega(x_{\alpha_2=(i,j)}(1))+val_p(x_{i,j});(\text{if} \ i<j)\\
			V_{i,j}&=\omega(h_{\delta=(i,i+1)}(1+p))+val_p(x_{i,i}-x_{i-1,i-1});(\text{if} \ i=j).
			\end{align*}
		\end{proposition}
		\begin{proof}
			
		To prove equation $\ref{eq:Sq}$, we first note that  for all $k=1,...,m=\min\{i-1,j-1\}$, we have $k<i,k<j$, hence (cf. equation $\ref{eq:orderequation}$) the terms $\omega(x_{\beta_1=(i,k)}(p))+val_p(x_{i,k})=\frac{k-i}{n}+val_p(px_{i,k})$ and $\omega(x_{\alpha_1=(k,j)}(1))+val_p(x_{k,j})=\frac{j-k}{n}+val_p(x_{k,j})$ belong to R.H.S of equation $\ref{eq:Sq-1}$ (because since $k<i$ and $k<j$,  equation $\ref{eq:orderequation}$ gives $v_{i,k}<v_{i,j}$ and $v_{k,j}<v_{i,j}$). Thus, to prove equation $\ref{eq:Sq}$, it suffices to find 
		{
		\begin{equation}\label{eq:findminimum}
		\min\limits_{\substack{k=1,...,m}}\Big\{ \frac{k-i}{n}+val_p(px_{i,k}),\frac{j-k}{n}+val_p(x_{k,j}),v_{i,j}\Big\}
		\end{equation}
		and show that 
		\begin{align*}
		&\min\limits_{\substack{k=1,...,m}}\Big\{ \frac{k-i}{n}+val_p(px_{i,k}),\frac{j-k}{n}+val_p(x_{k,j}),v_{i,j}\Big\}\\
		&=\min\limits_{\substack{\alpha_1=(k,j),\beta_1=(i,k) \\ k=1,...,m}}\Big\{ \omega(x_{\alpha_1=(k,j)}(1))+val_p(x_{k,j}),\omega(x_{\beta_1=(i,k)}(p))+val_p(x_{i,k}),V_{i,j}\Big\}.
		\end{align*} 
	}
		We divide this problem in three cases, first when $i>j$, second when $i<j$ and third when $i=j$. These three cases are dealt in \textbf{(I),(II),(III)} below.\\

		First we consider the case $i,j \geq 2,i >j, m:=\min\{i-1,j-1\}=j-1$. So $m<j<i$. Then
		{
		\begin{flushleft}
			$\min\limits_{\substack{k=1,...,m}}\Big\{ \frac{k-i}{n}+val_p(px_{i,k}),\frac{j-k}{n}+val_p(x_{k,j}),v_{i,j}=\frac{j-i}{n}+val_p(a_{i,j})\Big\}$\\
			$=\min\limits_{\substack{k=1,...,m}}\Big\{ \frac{k-i}{n}+1+val_p(x_{i,k}),\frac{j-k}{n}+val_p(x_{k,j}),\frac{j-i}{n}+val_p((\sum_{k=1}^m px_{i,k}x_{k,j})+px_{i,j})\Big\}$\\
			$=\min\limits_{\substack{\beta_1=(i,k),\alpha_1=(k,j) \\ k=1,...m}}\Big\{\omega(x_{\beta_1}(p))+val_p(x_{i,k}),\omega(x_{\alpha_1}(1))+val_p(x_{k,j}),\frac{j-i}{n}+val_p((\sum_{k=1}^m px_{i,k}x_{k,j})+px_{i,j})\Big\}$ \\
			$=\min\limits_{\substack{\beta_1=(i,k),\alpha_1=(k,j)\\\beta_2=(i,j),  k=1,...m}}\Big\{\omega(x_{\beta_1}(p))+val_p(x_{i,k}),\omega(x_{\alpha_1}(1))+val_p(x_{k,j}),\omega(x_{\beta_2}(p))+val_p(x_{i,j})\Big\}$\textbf{---(I)},
		\end{flushleft}
	}
		where the first equality follows from $\ref{2.2.21}$ and the second equality follows from $(\ref{2.2.1}-\ref{2.2.3})$. To prove the last equality, we notice that if $val_p(px_{i,j})<val_p(\sum_{k=1}^m px_{i,k}x_{k,j})$ then it is obvious. On the other hand, if 
		\begin{equation}\label{eq:eatra}
		val_p(px_{i,j}) \geq val_p(\sum_{k=1}^m px_{i,k}x_{k,j})
		\end{equation}
		then both the terms (L.H.S and R.H.S of the last equality of \textbf{I}) equals 
		\begin{center}
			$\min\limits_{\substack{\beta_1=(i,k),\alpha_1=(k,j) \\ k=1,...m}}\Big\{\omega(x_{\beta_1}(p))+val_p(x_{i,k}),\omega(x_{\alpha_1}(1))+val_p(x_{k,j})\Big\}$,
		\end{center}
		because 
		{
		\begin{align}
		\frac{j-i}{n}+val_p((\sum_{k=1}^m px_{i,k}x_{k,j})+px_{i,j}) &\geq \frac{j-i}{n}+val_p(\sum_{k=1}^m px_{i,k}x_{k,j}) \ (\text{using} \ (\ref{eq:eatra}))\label{2.2.33}\\
		&\geq \frac{j-i}{n}+\min\limits_{\substack{k=1,...,m}}\{val_p(px_{i,k})\}\label{2.2.34}\\
		&= \min\limits_{\substack{k=1,...,m}}\Big\{\frac{j-i}{n}+val_p(px_{i,k})\Big\}\label{2.2.35}\\
		&\geq \min\limits_{\substack{k=1,...,m}}\Big\{\frac{k-i}{n}+val_p(px_{i,k})\Big\} \ (\text{as} \ k\leq m=j-1<j)\label{2.2.36}\\
		&=\min_{k=1,...,m}\Big\{\omega(x_{\beta_1=(i,k)}(p))+val_p(x_{i,k})\Big\}\label{2.2.37}.
		\end{align}
		and
		\begin{align*}
		\omega(x_{\beta_2=(i,j)}(p))+val_p(x_{i,j})&=\frac{j-i}{n}+val_p(px_{i,j})\\
		&\geq \frac{j-i}{n}+val_p(\sum_{k=1}^m px_{i,k}x_{k,j}) \ (\text{using} \ (\ref{eq:eatra}))\\
		& \geq \min_{k=1,...,m}\Big\{\omega(x_{\beta_1=(i,k)}(p))+val_p(x_{i,k})\Big\} \ (\text{using} \  (\ref{2.2.33}-\ref{2.2.37})).
		\end{align*}
	}
		The argument for the case $i,j \geq 2,i<j,m=i-1<i<j$ is similar to our previous computation and so we omit it. The result that we obtain in that case is the following:
		{
		\begin{flushleft}
			$\min\limits_{\substack{k=1,...,m}}\Big\{ \frac{k-i}{n}+val_p(px_{i,k}),\frac{j-k}{n}+val_p(x_{k,j}),v_{i,j}=\frac{j-i}{n}+val_p(a_{i,j})\Big\}$\\
			$=\min\limits_{\substack{\beta_1=(i,k),\alpha_1=(k,j) \\ k=1,...m}}\Big\{\omega(x_{\beta_1}(p))+val_p(x_{i,k}),\omega(x_{\alpha_1}(1))+val_p(x_{k,j}),\frac{j-i}{n}+val_p((\sum_{k=1}^m px_{i,k}x_{k,j})+x_{i,j})\Big\}$\\
			$=\min\limits_{\substack{\beta_1=(i,k),\alpha_1=(k,j)\\ \alpha_2=(i,j),  k=1,...m}}\Big\{\omega(x_{\beta_1}(p))+val_p(x_{i,k}),\omega(x_{\alpha_1}(1))+val_p(x_{k,j}),\omega(x_{\alpha_2}(1))+val_p(x_{i,j})\Big\}$\textbf{---(II)}.
		\end{flushleft}
		Now, for $i,j \geq 2,i=j,m:=\min\{i-1,j-1\}=i-1=j-1$,
		\begin{flushleft}
			$\min\limits_{\substack{k=1,...,m}}\Big\{ \frac{k-i}{n}+val_p(px_{i,k}),\frac{j-k}{n}+val_p(x_{k,j}),v_{i,i}=val_p(a_{i,i}-1)\Big\}$\\
			$=\min\limits_{\substack{k=1,...,m}}\Big\{ \frac{k-i}{n}+val_p(px_{i,k}),\frac{j-k}{n}+val_p(x_{k,j}),val_p((\sum_{k=1}^m px_{i,k}x_{k,i})+p(x_{i,i}-x_{i-1,i-1}))\Big\}$\\
			$=\min\limits_{\substack{k=1,...,m}}\Big\{ \frac{k-i}{n}+val_p(px_{i,k}),\frac{j-k}{n}+val_p(x_{k,j}),1+val_p(x_{i,i}-x_{i-1,i-1})\Big\}$\\
			$=\min\limits_{\substack{\beta_1=(i,k),\alpha_1=(k,j),\\ \delta=(i,i+1), k=1,...m}}\Big\{\omega(x_{\beta_1}(p))+val_p(x_{i,k}),\omega(x_{\alpha_1}(1))+val_p(x_{k,j}),\omega(h_{\delta}(1+p))+val_p(x_{i,i}-x_{i-1,i-1})\Big\},$
		\end{flushleft}
		\hfill\(\textbf{----(III)}\)  \\
		
	}
		
		where the first equality follows from $\ref{2.2.23}$ and the third equality follows from $\ref{2.2.2}$. To prove the second equality, we notice that if $val_p(p(x_{i,i}-x_{i-1,i-1}))<val_p(\sum_{k=1}^m px_{i,k}x_{k,i})$ then it is obvious. On the contrary, if
		\begin{equation}
		\label{eq:extra}
		val_p(p(x_{i,i}-x_{i-1,i-1}))\geq val_p(\sum_{k=1}^m px_{i,k}x_{k,i})
		\end{equation}
		then both the terms (L.H.S and R.H.S of the second equality of \textbf{III}) equals 
		\begin{center}
			$\min\limits_{\substack{k=1,...,m}}\Big\{ \frac{k-i}{n}+val_p(px_{i,k}),\frac{j-k}{n}+val_p(x_{k,j})\Big\}$
		\end{center}
		as 
		{
		\begin{align}
		val_p((\sum_{k=1}^m px_{i,k}x_{k,i})+p(x_{i,i}-x_{i-1,i-1})) &\geq val_p(\sum_{k=1}^m px_{i,k}x_{k,i}) \ (\text{using} \ (\ref{eq:extra})) \label{2.2.39}\\
		& \geq \min\limits_{\substack{k=1,...,m}}\Big\{val_p(px_{i,k})\Big\}\label{2.2.40}\\
		&> \min\limits_{\substack{k=1,...,m}}\Big\{\frac{k-i}{n}+val_p(px_{i,k})\Big\} \ (\text{as} \ k<i)\label{2.2.41}
		\end{align}
		and 
		\begin{align*}
		1+val_p((x_{i,i}-x_{i-1,i-1})) &\geq  val_p(\sum_{k=1}^m px_{i,k}x_{k,i}) \ (\text{using} \ (\ref{eq:extra})) \\
		&> \min\limits_{\substack{k=1,...,m}}\Big\{\frac{k-i}{n}+val_p(px_{i,k})\Big\} \ (\text{using} \ (\ref{2.2.39}-\ref{2.2.41}))
		\end{align*}
	}
		This completes the demonstration of equation  $\ref{eq:Sq}$.

		\end{proof}

\begin{remark}\label{rem:ordered}
	Note that in theorem \ref{eq:orderedbasisiwahori}, we have ordered the lower unipotent matrices by the increasing height function on the roots. Theorem \ref{eq:orderedbasisiwahori} also holds true (with the same proof) if we order the lower unipotent matrices by the following (lexicographic) order:
	
	the matrix $(1+pE_{i,j})$ comes before $(1+pE_{k,l})$ if and only if $i<k$ and $i=k \implies j<l$,
	
	that is, we start with the top and left extreme and the fill the lines from the left, going down.
	
	With $x_{k,1} \in \mathbb{Z}_p$, 
	\begin{equation*}
	(1+px_{2,1}E_{2,1})(1+px_{3,1}E_{3,1})(1+px_{3,2}E_{3,2})\cdots (1+px_{n,n-1}E_{n,n-1})=I+\sum\limits_{\substack{i\in [2,n] \\ j\in [1,i-1]}}px_{i,j}E_{i,j},
	\end{equation*}
	(this is $X$ in equation \ref{eq:matrixX}). Indeed, 
	for any $b\in \mathbb{N}$, if we take the product of the first $b$ terms in the L.H.S of the above equation, the set of column entries which appears in the subscript of the elementary matrices of the product, is disjoint from the row entry which appears in the subscript of the elementary matrix occurring in the $(b+1)^{th}$ term of the product in the L.H.S. We also use $E_{i,j}E_{k,l}=E_{i,l}$ for $j=k$ and $E_{i,j}E_{k,l}=0$ if $j \neq k$.
	
	The rest of the proof of Theorem \ref{eq:orderedbasisiwahori} goes without any change.
\end{remark}

\section{Relations in the Iwasawa algebra}\label{sub:two}
In this section we first recall the notion of the Iwasawa algebra of $G$ which is naturally  isomorphic as a $\mathbb{Z}_p$-module to a commutative ring of power series in several  variable over $\mathbb{Z}_p$.
Then we find the relations between those variables in order to give a ring theoretic presentation of the Iwasawa algebra (lemma \ref{relations}). 

This section is organized as follows. In section $\ref{sub:Iwasawa}$, we recall the notion of the Iwasawa algebra of $G$, and for $\beta \in \Phi^-,\alpha \in \Phi^+, \delta \in \Pi$, we identify (as a $\mathbb{Z}_p$-module) the Iwasawa algebra of $G$ with the ring of power series in the variables $U_{\beta},V_{\alpha},W_{\delta}$ over $\mathbb{Z}_p$. This isomorphism is given by sending $1+V_{\alpha}\mapsto x_{\alpha}(1), 1+W_{\delta} \mapsto h_{\delta}(1+p)$ and $1+U_{\beta} \mapsto x_{\beta}(p)$. As the Iwasawa algebra is non-commutative, this isomorphism is obviously not a ring isomorphism. Therefore, in section $\ref{sub:relationsinIwahori}$, we study the products of the variables "in the wrong order", viewing them as elements of the Iwasawa algebra and then we find the relations among them $(\ref{eq:firstIwahori}-\ref{eq:fifteenIwahori})$. Finally,  we consider $\mathcal{A}$ to be the non-commutative power series over $\mathbb{Z}_p$ in the variables $V_{\alpha}$, $W_{\delta}$, and $U_{\beta}$ corresponding to the order of roots as in Theorem \ref{eq:orderedbasisiwahori} and construct a natural map $\varphi:\mathcal{A} \rightarrow \Lambda(G)$ which  factors through the quotient of $\mathcal{A}$ modulo the relations $(\ref{eq:firstIwahori}-\ref{eq:fifteenIwahori})$.

\subsection{Iwasawa algebra}\label{sub:Iwasawa}
After Theorem $\ref{eq:orderedbasisiwahori}$, we have a homeomorphism $c:\mathbb{Z}_p^d\rightarrow G$. Let $C(G)$ be continuous functions from $G$ to $\mathbb{Z}_p$. The map $c$ induces, by pulling back functions, an isomorphism of $\mathbb{Z}_p$-modules
\begin{center}
	$c^*:C(G) \simeq C(\mathbb{Z}_p^d)$.
\end{center}
\begin{definition}
Let $\Lambda(G)$  be the Iwasawa algebra of $G$ over $\mathbb{Z}_p$, that is, $$\Lambda (G):=\varprojlim _{N\in \mathcal{N}(G)}(G/N),$$ where $\mathcal{N}(G)$ is the set of open normal subgroups in $G$.	
\end{definition}
 Now, lemma 22.1 of \cite{Lie} shows that 
\begin{center}
	$\Lambda (G)=\text{Hom}_{\mathbb{Z}_p}(C(G),\mathbb{Z}_p)$.
\end{center}
So, dualizing the map $c^*$, we get an isomorphism of $\mathbb{Z}_p$-modules 
\begin{center}
	$c_{*}=\Lambda (\mathbb{Z}_p^d)\cong \Lambda (G)$.
\end{center}
\begin{lemma}[Prop. $20.1$ of $ \cite{Lie}$]\label{ctilde}
	We have the following isomorphism of $\mathbb{Z}_p$-modules
	\begin{center}
		$\tilde{c}:\mathbb{Z}_p[[U_{\beta},V_{\alpha},W_{\delta}; \beta \in \Phi^-,\alpha \in \Phi^+, \delta \in \Pi]]\cong \Lambda(G)$\\
		$1+V_{\alpha}\longmapsto x_{\alpha}(1)$\\
		$1+W_{\delta} \longmapsto h_{\delta}(1+p)$\\
		$1+U_{\beta} \longmapsto x_{\beta}(p)$.
	\end{center}
\end{lemma}

We note that for obvious reasons the isomorphism above is not a ring isomorphism. Now, let $b_i:=g_i-1$ and $b^m:=b_1^{m_1}\cdots b_d^{m_d}$ for any multi-index $m=(m_1,...,m_d)\in \mathbb{N}^d$. 

\begin{definition}
	We can define a valuation on $\Lambda(G)$
	\begin{center}
	$\tilde{\omega}:=\tilde{\omega}_{\frac{1}{n}}:\Lambda(G)\backslash \{0\}\rightarrow [0,\infty)$
\end{center}
by
\begin{center}
	$\tilde{\omega}\Big(\sum_{m \in \mathbb{N}^d} c_mb^m\Big):=\inf_{m \in \mathbb{N}^d} \Big(val_p(c_m)+\sum_{i=1}^dm_i\tilde{\omega}(g_i)\Big)$
\end{center}
where $c_m \in \mathbb{Z}_p,$ with the convention that $\tilde{\omega}(0):=\infty$. 
	\end{definition}
	The valuation $\tilde{\omega}$ is the natural valuation on $\Lambda(G)$ (cf. \cite[Section 28]{Lie}).

\subsection{Relations}\label{sub:relationsinIwahori}  

The isomorphism $\tilde{c}$ of lemma \ref{ctilde}, provides an identification of the variables $V_{\alpha}, W_{\delta}, U_{\beta}$ as the elements of the Iwasawa algebra of $G$. Our objective is to find the relations among the above variables. For this we  use the Chevalley relations \cite{Yale} and sometimes also direct computation.

\begin{lemma}\label{relations}
	In the Iwasawa algebra $\Lambda(G)$, the variables $V_{\alpha}, W_{\delta}, U_{\beta}$ satisfy the following relations.
	{
	\begin{align}
	(1+W_{\delta})(1+U_{\beta})&=(1+U_{\beta})^q(1+W_{\delta}),(\beta \in \Phi^-), q = (1+p)^{\langle \beta, \delta \rangle}\label{eq:firstIwahori}\\
	(1+W_{\delta})(1+V_{\alpha})&=(1+V_{\alpha})^{q^{'}}(1+W_{\delta}),(\alpha \in \Phi^+),  q^{'} = (1+p)^{\langle \alpha, \delta \rangle}\label{eq:secondIwahori}\\
	V_{\alpha}U_{\beta}&=U_{\beta}V_{\alpha},(\alpha \in \Phi^+,\beta \in \Phi^-, \alpha \neq -\beta, \alpha+\beta \notin \Phi)\label{eq:thirdIwahori}\\
	(1+V_{\alpha})(1+U_{\beta})&=(1+V_{(i,k)})^p(1+U_{\beta})(1+V_{\alpha}),i<k,(\alpha =(i,j)\in \Phi^+,\beta=(j,k) \in \Phi^-) \label{eq:fiveiwahori}\\
	(1+V_{\alpha})(1+U_{\beta})&=(1+U_{(i,k)})(1+U_{\beta})(1+V_{\alpha}),i>k,(\alpha =(i,j)\in \Phi^+,\beta=(j,k) \in \Phi^-)\label{eq:sixIwahori}\\
	(1+V_{\alpha})(1+U_{\beta})&=(1+V_{(k,j)})^{-p}(1+U_{\beta})(1+V_{\alpha}),k<j,(\alpha =(i,j)\in \Phi^+,\beta=(k,i) \in \Phi^-)\label{eq:sevenIwahori}\\
	(1+V_{\alpha})(1+U_{\beta})&=(1+U_{(k,j)})^{-1}(1+U_{\beta})(1+V_{\alpha}),k>j,(\alpha=(i,j) \in \Phi^+,\beta=(k,i) \in \Phi^-) \label{eightIwahori}\\
	(1+V_{\alpha})(1+U_{-\alpha})&=(1+U_{-\alpha})^{(1+p)^{-1}}(1+W_{(i,i+1)})\cdots (1+W_{(j-1,j)})(1+V_{\alpha})^{(1+p)^{-1}},(\alpha=(i,j) \in \Phi^+) \label{eq:fourIwahori}\\
	U_{\beta_1}U_{\beta_2}&=U_{\beta_2}U_{\beta_1},(\beta_1,\beta_2 \in \Phi^-, \beta_1+\beta_2 \notin \Phi) \label{eq:nineIwahori}\\
	(1+U_{\beta_1})(1+U_{\beta_2})&=(1+U_{(i,k)})^p(1+U_{\beta_2})(1+U_{\beta_1}),(\beta_1=(i,j),\beta_2=(j,k) \in \Phi^-)\label{eq:tenIwahori}\\
	(1+U_{\beta_1})(1+U_{\beta_2})&=(1+U_{(k,j)})^{-p}(1+U_{\beta_2})(1+U_{\beta_1}),(\beta_1=(i,j),\beta_2=(k,i) \in \Phi^-) \label{eq:elevenIwahori}\\
	W_{\delta_1}W_{\delta_2}&=W_{\delta_2}W_{\delta_1},(\delta_1,\delta_2 \in \Pi, \delta_1 \neq \delta_2)\label{eq:twelveIwahori}\\
	V_{\alpha_1}V_{\alpha_2}&=V_{\alpha_2}V_{\alpha_1},(\alpha_1,\alpha_2 \in \Phi^+, \alpha_1+\alpha_2 \notin \Phi)\label{eq:thirteenIwahori}\\
	(1+V_{\alpha_1})(1+V_{\alpha_2})&=(1+V_{(i,k)})(1+V_{\alpha_2})(1+V_{\alpha_1}),(\alpha_1=(i,j),\alpha_2=(j,k) \in \Phi^+)\label{eq:fourteenIwahori}\\
	(1+V_{\alpha_2})(1+V_{\alpha_1})&=(1+V_{(k,j)})(1+V_{\alpha_1})(1+V_{\alpha_2}),(\alpha_1=(i,j),\alpha_2=(k,i) \in \Phi^+) \label{eq:fifteenIwahori}.
	\end{align}
}
\end{lemma} 

\begin{proof}
	We use the Steinberg's relations in the group $G$ and then translate them in $\mathcal{A}$ in order to deduce the relations in lemma \ref{relations}. For example, recalling the notations $h_{\delta}(1+p)$ and $x_{\beta}(p)$ in section \ref{sub:ordiwa}, Steinberg \cite{Yale} gives
	
	\begin{center}
		$h_{\delta}(1+p)x_{\beta}(p)h_{\delta}(1+p)^{-1}=x_{\beta}((1+p)^{\langle \beta , \delta \rangle }p)$, $(\beta\in \Phi^-,\delta \in \Pi)$,
	\end{center}
	where $\langle \beta , \delta \rangle \in \mathbb{Z}$ (cf. p. $30$ of \cite{Yale}).
	So the corresponding relation in the Iwasawa algebra is 
	\begin{equation}
	(1+W_{\delta})(1+U_{\beta})=(1+U_{\beta})^q(1+W_{\delta}),(\beta \in \Phi^-),
	\end{equation}
	where $q = (1+p)^{\langle \beta, \delta \rangle}$. This is relation \ref{eq:firstIwahori}\\
	
	By similar means we compute the other relations of lemma \ref{relations}. One can find the computation in section \ref{appendixCrelations}.
\end{proof}

We consider $\mathcal{A}$ - the universal  $p$-adic algebra  of non-commutative power series in the variables $V_{\alpha}$, $W_{\delta}$, and $U_{\beta}$ where $\alpha$ varies over the positive roots, $\delta$ varies over the simple roots and $\beta$ varies over the negative roots. We denote $\mathcal{A}=\mathbb{Z}_p\{\{V_{\alpha},W_{\delta},U_{\beta}, \alpha \in \Phi^+,\delta \in \Pi, \beta \in \Phi^-\}\}$. Thus, it 
is composed of all non-commutative series 
\begin{center}
	$f=\sum_{k\geqslant 0}\sum_ia_ix^i$
\end{center}
where $a_i \in \mathbb{Z}_p$ and, for all $k \geqslant 0$, $i$ runs over all maps
$\{1,2,...,k\}\rightarrow \{1,2,...,d\}$, ($d=|\Phi|+|\Pi|$); we set $x_1=U_{(n,1)},x_2=U_{(n-1,1)},x_3=U_{(n,2)},\ldots , x_{\frac{n(n-1)}{2}+1}=W_{(1,2)},\ldots, x_{\frac{n^2+n-2}{2}}=W_{(n-1,n)}, x_{\frac{n^2+n}{2}}=V_{(n-1,n)},\ldots, x_{d}=V_{(1,2)}$, i.e. $x_i$'s (with $i$ increasing) are assigned to the variables among $\{V_{\alpha},W_{\delta},U_{\beta}\}$ corresponding to the order as prescribed in Theorem \ref{eq:orderedbasisiwahori}. We set $x^i=x_{i_1}x_{i_2}\cdots x_{i_k}$.

Let $\mathcal{R}\subset \mathcal{A}$ be the closed (two-sided) ideal generated by the relations 
(\ref{eq:firstIwahori}-\ref{eq:fifteenIwahori}) and $\overline{\mathcal{A}}$ be the reduction modulo $p$ of $\mathcal{A}$. After lemma \cite[lemma 1.3]{Laurent}, we obtain
\begin{corollary}\label{eq:lemmaunimp}
	Let $\overline{\mathcal{R}}$ be the image of $\mathcal{R}$ in $\overline{\mathcal{A}}$. Then $\overline{\mathcal{R}}$ is the closed two-sided ideal in $\overline{\mathcal{A}}$ generated by the  relations (\ref{eq:firstIwahori}-\ref{eq:fifteenIwahori}).
\end{corollary}

\section{Presentation of the Iwasawa algebra $\Lambda(G)$ for $p>n+1$}
Our goal in this section is to give an explicit presentation of $\Lambda(G)$. The strategy of the proof is to show the corresponding statement for 
\begin{center}
	$\Omega_G:=\Lambda_G \otimes_{\mathbb{Z}_p}\mathbb{F}_p$,
\end{center}
which is the Iwasawa algebra with finite coefficients. We show in Theorem \ref{eq:secondimptheoremIwahori} that for $p>n+1$, the Iwasawa algebra $\Omega_G$ is naturally isomorphic to $\overline{\mathcal{A}}/\overline{\mathcal{R}}$. In section \ref{sub:conti}, we construct a continuous surjective map $\varphi:\mathcal{A} \rightarrow \Lambda(G)$. Thus, we get a natural continuous surjection $\mathcal{B}:=\mathcal{A}/\mathcal{R} \rightarrow \Lambda(G)$. Therefore, by lemma $\ref{eq:lemmaunimp}$,  we obtain a surjection $\overline{\varphi}:\overline{\mathcal{B}}:=\overline{\mathcal{A}}/\overline{\mathcal{R}} \rightarrow \Omega_G$. In section $\ref{sub:longlemmaiwahori}$, using the natural grading on $\overline{\mathcal{B}}$, we show that $\dim $ gr$^m\overline{\mathcal{B}} \leqslant \dim_{\mathbb{F}_p} \text{gr}^m\Omega_G$ (cf. Proposition $\ref{eq:longlemma}$). Finally, in section $\ref{sub:maintheoremsiwahori}$, we give the proof of an explicit presentation of $\Lambda(G)$ and $\Omega_G$ (cf. Theorems $\ref{eq:secondimptheoremIwahori}$ and $\ref{eq:maintheoremIwahori}$). Later in corollary \ref{eq:theoremgln} we extend our result to obtain an explicit presentation of the Iwasawa algebra for the pro-$p$ Iwahori of $GL(n,\mathbb{Z}_p)$.\\

\subsection{Iwasawa algebra with finite coefficients}\label{sub:conti}
We first construct a natural surjective, continuous map $\varphi: \mathcal{A} \rightarrow \Lambda(G)$.

The non-commutative polynomial algebra $$A:=\mathbb{Z}_p\{x_1,...,x_d\}$$   is a dense subalgebra of $\mathcal{A}$.
\begin{lemma}
	Let us define a natural map $\varphi:A \rightarrow \Lambda(G)$ mapping $x_i \in {A}$ (with $i$ increasing) to the corresponding variable among $\{V_{\alpha},W_{\delta},U_{\beta}\}$ in the Iwasawa algebra $\Lambda(G)$, according to the order as prescribed in Theorem \ref{eq:orderedbasisiwahori}. Then, this map extends continuously to a surjective homomorphism $\mathcal{A}\rightarrow \Lambda(G)$.
\end{lemma}
\begin{proof}
	It suffices to show that if a sequence in $\mathcal{A}$ converges to $0$ in the topology of $\mathcal{A}$, the image converges to $0$ in $\Lambda(G)$. \\
	The topology of $\mathcal{A}$ is given by the valuation $v_{\mathcal{A}}$; we write 
	\begin{center}
		$F=\sum a_ix^i$\\
		$v_{\mathcal{A}}(F)=\inf_{i}(val_p(a_i)+|i|)$.
	\end{center}
	Lazard (cf. $2.3$, chap.$3$ of \cite{Lazard}) shows that $\omega$ is an additive valuation, $\omega(\eta\nu)=\omega(\nu)+\omega(\eta)$ for $\nu,\eta \in \Lambda(G)$. For $F \in \mathcal{A}$, with image $\mu \in \Lambda(G)$, we have 
	\begin{align*}
	\omega(\mu)&=\omega(\sum a_ix^i)\\
	&\geqslant \inf_i\{val_p(a_i)+\sum_{i=1}^dm_i\omega(g_i)\} \ (\text{for} \ \text{some} \ m_i \in \mathbb{N})\\
	&\geqslant \frac{1}{n}v_{\mathcal{A}}(F).
	\end{align*}
	
	The third inequality follows because\\ $\omega(W_{(i^{'},i^{'}+1)})=1>\frac{1}{n}$, 
	$\omega(U_{(i,j)})=\frac{n-i+j}{n}\geqslant \frac{1}{n}$ and 
	$\omega(V_{(i^{''},j^{''})})=\frac{j^{''}-i^{''}}{n}\geqslant\frac{1}{n}$,
	where $(i,j)\in \Phi^-,i^{'}=1,...,n-1,(i^{''},j^{''})\in \Phi^+$. \\
	
	Hence, the map $\varphi:\mathcal{A}\rightarrow \Lambda(G)$ is continuous. 
	
	The surjectivity follows from  the fact that $\varphi$ is already surjective if $\mathcal{A}$ is replaced by the set of linear combinations of well-ordered monomials ($i$ increasing).
	
\end{proof}

The Iwasawa algebra $\Lambda(G)$ is filtered by $\frac{1}{n}\mathbb{N}$. The filtration of $\Lambda(G)$ defines a filtration on  $\Omega_G:=\Lambda_G \otimes_{\mathbb{Z}_p}\mathbb{F}_p$ given by $F^\upsilon\Omega_G=F^\upsilon\Lambda_G\otimes \mathbb{F}_p$. Reduction modulo $p$ gives is a natural map $\overline{\mathcal{A}}\rightarrow \Omega_G$ whose kernel contains $\overline{\mathcal{R}}$.\\

 As an $\mathbb{F}_p$-vector space, gr$^v\Lambda_G$ is generated by the independent elements 
\begin{center}
	$p^d\prod_{(i,j)\in \Phi^-}U_{(i,j)}^{p_{ij}}\prod_{i^{'}=1}^{n-1}W_{(i^{'},i^{'}+1)}^{m_{i^{'},i^{'}+1}}\prod_{(i^{''},j^{''})\in \Phi^+}V_{(i^{''},j^{''})}^{n_{i^{''},j^{''}}}$ such that $d+\sum_{(i,j)\in \Phi^-}[\frac{j-i}{n}+1]p_{ij}+\sum_{i^{'}=1}^{n-1}{m_{i^{'},i^{'}+1}}+\sum_{(i^{''},j^{''})\in \Phi^+}(\frac{j^{''}-i^{''}}{n}){n_{i^{''},j^{''}}}=v$.
\end{center}
(Cf. p.$199$ of \cite{Lie}). Then  gr$^v\Omega_G$ is generated by the elements
\begin{center}
	$\prod_{(i,j)\in \Phi^-}U_{(i,j)}^{p_{ij}}\prod_{i^{'}=1}^{n-1}W_{(i^{'},i^{'}+1)}^{m_{i^{'},i^{'}+1}}\prod_{(i^{''},j^{''})\in \Phi^+}V_{(i^{''},j^{''})}^{n_{i^{''},j^{''}}}$ such that $\sum_{(i,j)\in \Phi^-}[\frac{j-i}{n}+1]p_{ij}+\sum_{i^{'}=1}^{n-1}{m_{i^{'},i^{'}+1}}+\sum_{(i^{''},j^{''})\in \Phi^+}(\frac{j^{''}-i^{''}}{n}){n_{i^{''},j^{''}}}=v$,
\end{center}
and these elements are linearly independent. We do not change the filtration replacing $v \in \frac{1}{n}\mathbb{N}$ by $m=nv \in \mathbb{N}$, the valuations of the variables $(U_{(i,j)},W_{(i^{'},i^{'}+1)},V_{(i^{''},j^{''})})$ being now $(n-i+j,n,j^{''}-i^{''})$ where $(i,j)\in \Phi^-,i^{'}=1,...,n-1,(i^{''},j^{''})\in \Phi^+$. \\

In particular we have for $m \in \mathbb{N}:$ 
\begin{thm}\label{eq:bird}
	(Lazard) The dimension $d_m$ of gr$^m\Omega_G$ over $\mathbb{F}_p$ is equal to the dimension of the space of homogeneous symmetric polynomials of degree $m$ in the variables $\{U_{(i,j)},W_{(i^{'},i^{'}+1)},V_{(i^{''},j^{''})}\}$ having degrees corresponding to their valuations $(n-i+j,n,j^{''}-i^{''})$ where $(i,j)\in \Phi^-,i^{'}\in [1,n-1],(i^{''},j^{''})\in \Phi^+$. 
\end{thm}

We must now consider on $\overline{\mathcal{A}}$, the filtration of $\overline{\mathcal{A}}$ obtained by assigning  the degrees $(n-i+j,n,j^{''}-i^{''})$ to the formal variables 
$(U_{(i,j)},W_{(i^{'},i^{'}+1)},V_{(i^{''},j^{''})})$. Then gr$^m \overline{\mathcal{A}}$ is isomorphic to the space of non-commutative polynomials of degree $m$. We endow $\overline{\mathcal{B}}=\overline{\mathcal{A}}/\overline{\mathcal{R}}$ with the induced filtration, so 
\begin{center}
	gr$^m\overline{\mathcal{B}}=Fil^m\overline{\mathcal{A}}/(Fil^{m+1}\overline{\mathcal{A}}+Fil^m\overline{\mathcal{A}}\cap \overline{\mathcal{R}})$.
\end{center}
By construction the map gr$^m\overline{\mathcal{B}} \rightarrow $ gr$^m\Omega_G$ is surjective. 

\subsection{Bound on the dimension of the graded pieces of $\mathcal{B}$ }\label{sub:longlemmaiwahori}
The following proposition gives an upper bound for $\dim $ gr$^m\overline{\mathcal{B}}$ generalizing the case for the pro-$p$ Iwahori of $SL(2,\mathbb{Z}_p)$ by Clozel (\cite[lemma 2.11]{Clozel2}).
\begin{proposition}
	\label{eq:longlemma}
	For $m\geqslant 0$, we have $\dim $ gr$^m\overline{\mathcal{B}} \leqslant d_m=\dim_{\mathbb{F}_p} \text{gr}^m\Omega_G $.
\end{proposition}
\begin{proof}
	The Lemma is true for $m=0$. For $m=1$, gr$^1\overline{\mathcal{B}}$ is a quotient of the space $Fil^1\overline{\mathcal{A}}/Fil^2\overline{\mathcal{A}}$ with basis $U_{(n,1)}$ and $V_{(i_1,i_1+1)}$ for $i_1=1,2,...,n-1$. So, $\dim $ gr$^1\overline{\mathcal{B}}\leqslant n=d_1$.\\
	
	To show the general case we will consider each relation and then apply induction argument to show that we decrease the number of inversions.\\
	
	First consider relation $\ref{eq:firstIwahori}$\\
	\begin{center}
		$(1+W_{\delta})(1+U_{\beta})=(1+U_{\beta})^q(1+W_{\delta}),(\delta \in \Pi, \beta \in \Phi^-)$,
	\end{center}
	where $q=(1+p)^{\langle \beta , \delta \rangle } \equiv 1[p]$. We have 
	\begin{center}
		$(1+U_{\beta})^q=1+qU_{\beta}+\frac{q(q-1)}{2}U_{\beta}^2+\cdots $
	\end{center}
	where $\frac{q(q-1)}{2}\equiv 0[p]$. By the Lazard condition $p>n+1$, we can show that
	\begin{center}
		$(1+W_{\delta})(1+U_{\beta})=(1+U_{\beta})(1+W_{\delta})$( mod $Fil^{n+s+1}$),
	\end{center} 
	where $s=$ degree of $U_{\beta}$. This is because, for any natural number $m \geq 2$, $U_{\beta}^m$ has degree $ms$. So, we need to show that 
	\begin{center}
		$ms\leq n+s$ implies ${q \choose m}\equiv 0$ (mod $p$),
	\end{center}
	
	\begin{flushleft}
		i.e. $\hspace{4cm}$ \ $(m-1)s\leq n$ implies ${q \choose m}\equiv 0$ (mod $p$).
	\end{flushleft}
	Now, $(m-1) \leq (m-1)s $. Therefore, it suffices to show that 
	\begin{equation}\label{eq:equationimp}
	m-1 \leq n \implies {q \choose m}\equiv 0 \ (\text{mod} \  p).
	\end{equation}
	But by the Lazard condition we get $m-1 
	\leq n <p-1$. So, $m<p$ and then trivially $val_p(m!)=0$. Hence, ${q \choose m}\equiv 0$ (mod $p$)which gives\\
	
	\begin{center}
		$W_{\delta}U_{\beta}=U_{\beta}W_{\delta}$ in $Fil^{n+s+1},(\delta \in \Pi,\beta \in \Phi^-)$. 
	\end{center}
	Consider a non-commutative monomial 
	\begin{center}
		$x^i=x_{i_1}x_{i_2}\cdots x_{i_w}$.
	\end{center}
	
	To avoid confusion we will write $x^i$ as $x^{i^*}$ because we will be using $i,j,k$ for the roots.  Assume the homogeneous degree of $x^{i^*}$ is equal to $t$. We can change $x^{i^*}$ into a well-ordered monomial ($b \rightarrow i_{b}$ increasing ) by a sequence of transpositions (Lemma $3.2$ of \cite{Laurent}). Consider a move $(b,b+1) \rightarrow (b+1,b)$ and assume $i_{b}>i_{b+1}$. We write
	\begin{center}
		$x^{i^*}=x^fx_{b}x_{b+1}x^e$
	\end{center}
	where deg$(f)=r^{'}$, deg$(e)=s^{'}$, deg$(i^*)=t$. Henceforth, we fix the notations $r^{'},s^{'},t$ to be the degrees of $f,e,i^*$ respectively. If 
	\begin{center}
		$(x_{b},x_{b+1})=(W_{\delta},U_{\beta})$,
	\end{center}
	then $x^fU_{\beta}W_{\delta}x^e\equiv x^{i^*}$ (mod $Fil^{t+1}$),  $t=r^{'}+n+s+s^{'}$. This reduces the number of inversions in $x^{i^*}$. \\
	
	We  do the same argument for the other relations (\ref{eq:secondIwahori}-\ref{eq:fifteenIwahori}), i.e. we  consider each of the relations and show that we reduce the number of inversions in each case. The computations can be found in section \ref{appendixDlong}. This completes the proof Proposition \ref{eq:longlemma}, because note that, inside $gr^m\overline{\mathcal{B}}$, we can arrange the variables in the wrong order by a sequence of transposition to put them in the right order (the right order as in the algebra $\mathcal{A}$). This shows that $\dim \text { }  gr^m\overline{\mathcal{B}} \leq \dim \text{ } gr^m\Omega_G$ since, by theorem \ref{eq:bird}, gr$^m\Omega_G$ contains homogeneous \textit{symmetric} polynomials.

\end{proof}
\subsection{Explicit presentations of the Iwasawa algebras $\Lambda(G)$ and $\Omega_G$}\label{sub:maintheoremsiwahori}
In this section we give our main theorem constructing an explicit presentation of the Iwasawa algebras  $\Lambda(G)$ and $\Omega_G$. In corollary \ref{eq:theoremgln}, we extend our result to include the case of the pro-$p$ Iwahori of $GL(n,\mathbb{Z}_p)$. 
\begin{thm}
	\label{eq:secondimptheoremIwahori}
	The map $\overline{\mathcal{A}} \rightarrow \Omega_G$ gives an isomorphism $\overline{\mathcal{B}}:=\overline{\mathcal{A}}/\overline{\mathcal{R}} \cong \Omega_G$. 
\end{thm}
\begin{proof}
	(Cf. \cite{Laurent} and \cite{Clozel2}).  The natural map $\varphi: \mathcal{A} \rightarrow \Lambda(G)$ respects the filtration and reduces modulo $p$ to $\overline{\varphi}:\overline{\mathcal{B}} \rightarrow \Omega_G$.   As $\overline{\varphi}$ is surjective, the natural map of graded algebras
	\begin{center}
		gr$\overline{\varphi}:$ gr$^{\bullet}\overline{\mathcal{B}}\rightarrow $ gr$^{\bullet}\Omega_{G}$
	\end{center}
	is surjective. Moreover, it is an isomorphism because $\dim $ gr$^m\overline{\mathcal{B}}\leqslant \dim $ gr$^m\Omega_{G}$ (proposition \ref{eq:longlemma}). Since the filtration on $\overline{\mathcal{B}}$ is complete, we deduce that $\overline{\varphi}$ is an isomorphism (cf. Theorem $4$ ($5$), p.$31$ of \cite{Zaras}). We have $\overline{\mathcal{B}}$ complete because $\overline{\mathcal{B}}=\overline{\mathcal{A}}/\overline{\mathcal{R}}$, where $\overline{\mathcal{R}}$ is closed and therefore complete for the filtration induced from that of $\overline{\mathcal{A}}$.
\end{proof}
\begin{thm}\label{eq:maintheoremIwahori}
	The map $\mathcal{A}\rightarrow \Lambda(G)$ gives, by passing to the quotient, an isomorphism $\mathcal{B}=\mathcal{A}/\mathcal{R} \cong \Lambda(G)$.
\end{thm}
\begin{proof}
	 We need the argument of \cite{Laurent}.
	The reduction of $\varphi$ is $\overline{\varphi}$. We recall that $\overline{\mathcal{R}}$ is the image of $\mathcal{R}$ in $\overline{\mathcal{A}}$. Let $f \in \mathcal{A}$ satisfies $\varphi(f)=0$. We then have $\overline{f} \in \overline{\mathcal{R}}$ since $\overline{\mathcal{A}}/\overline{\mathcal{R}}\cong \Omega_{G}$. So $f=r_1+pf_1, r_1 \in \mathcal{R},f_1 \in \mathcal{A}.$ Then $\varphi(f_1)=0$. Inductively, we obtain an expression $f=r_n+p^nf_n$ of the same type. Since $p^nf_n\rightarrow 0$ in $\mathcal{A}$ and $\mathcal{R}$ is closed, we deduce that $f \in \mathcal{R}$. 
\end{proof}

This gives us the following corollary:
\begin{corollary}
	\label{eq:theoremgln}
	The Iwasawa algebra of the pro-$p$ Iwahori subgroup of $GL_n(\mathbb{Z}_p)$ is a quotient $\mathcal{A}^{\prime}/\mathcal{R}$, with 
	$\mathcal{A}^{\prime}=\mathbb{Z}_p\{\{Z,V_{\alpha},U_{\beta},W_{\delta}, \alpha \in \Phi^+,\beta\in \Phi^-,\delta \in \Pi\}\}$ and $\mathcal{R}$ is defined by the relations ($\ref{eq:firstIwahori}-\ref{eq:fifteenIwahori}$) and \\
	(Comm) $Z$ commutes with $U_{\beta},V_{\alpha},W_{\delta}$ for all $\alpha,\beta, \delta$.
\end{corollary}
Here the variable $Z$  corresponds to the element $(1+p)I_n\in GL_n(\mathbb{Z}_p)$. \\

In conclusion, for $p>n+1$, we have found a Lazard basis of $G$ with respect to its $p$-valuation $\omega$ (see theorem $\ref{eq:orderedbasisiwahori}$). Furthermore, we have obtained the relations inside the Iwasawa algebra of $G$, thus giving us an explicit presentation of $\Lambda(G)$ (see theorem \ref{eq:maintheoremIwahori}) by controlling the dimension of gr$^m\overline{\mathcal{B}}$ (Proposition \ref{eq:longlemma}). This readily gives the presentation of the Iwasawa algebra of the pro-$p$ Iwahori subgroup of $GL_n(\mathbb{Z}_p)$ (corollary \ref{eq:theoremgln}). Also, as noted in the introduction of this article, it is a future project of Dong Han and Feng Wei to use our results  to answer the open question on the existence of non-trivial normal elements in $\Omega_G$ (section $5$ of \cite{HW1}). Note that, in \cite{CorRay}, we have found out a minimal set of topological generators for the pro-$p$ Iwahori of any general reductive group over $\Z_p$ and a nice future project would be to find an explicit presentation of the Iwasawa algebras for general pro-$p$ Iwahori subgroups defined for any split reductive group.

\newpage

\newpage

\section{Computations for the proof of Lemma \ref{relations}}\label{appendixCrelations}

In this section we complete the computations needed for Lemma \ref{relations}. We quote lemma \ref{relations}.\\

\textbf{\textit{Lemma} \ref{relations}}	\textit{In the Iwasawa algebra} $\Lambda(G)$, \textit{the variables} $V_{\alpha}, W_{\delta}, U_{\beta}$ \textit{satisfy the following relations.}
{
	\begin{align*}
	(1+W_{\delta})(1+U_{\beta})&=(1+U_{\beta})^q(1+W_{\delta}),(\beta \in \Phi^-), q = (1+p)^{\langle \beta, \delta \rangle}\\
	(1+W_{\delta})(1+V_{\alpha})&=(1+V_{\alpha})^{q^{'}}(1+W_{\delta}),(\alpha \in \Phi^+),  q^{'} = (1+p)^{\langle \alpha, \delta \rangle}\\
	V_{\alpha}U_{\beta}&=U_{\beta}V_{\alpha},(\alpha \in \Phi^+,\beta \in \Phi^-, \alpha \neq -\beta, \alpha+\beta \notin \Phi)\\
	(1+V_{\alpha})(1+U_{\beta})&=(1+V_{(i,k)})^p(1+U_{\beta})(1+V_{\alpha}),i<k,(\alpha =(i,j)\in \Phi^+,\beta=(j,k) \in \Phi^-) \\
	(1+V_{\alpha})(1+U_{\beta})&=(1+U_{(i,k)})(1+U_{\beta})(1+V_{\alpha}),i>k,(\alpha =(i,j)\in \Phi^+,\beta=(j,k) \in \Phi^-)\\
	(1+V_{\alpha})(1+U_{\beta})&=(1+V_{(k,j)})^{-p}(1+U_{\beta})(1+V_{\alpha}),k<j,(\alpha =(i,j)\in \Phi^+,\beta=(k,i) \in \Phi^-)\\
	(1+V_{\alpha})(1+U_{\beta})&=(1+U_{(k,j)})^{-1}(1+U_{\beta})(1+V_{\alpha}),k>j,(\alpha=(i,j) \in \Phi^+,\beta=(k,i) \in \Phi^-) \\
	(1+V_{\alpha})(1+U_{-\alpha})&=(1+U_{-\alpha})^{(1+p)^{-1}}(1+W_{(i,i+1)})\cdots (1+W_{(j-1,j)})(1+V_{\alpha})^{(1+p)^{-1}},(\alpha=(i,j) \in \Phi^+) \\
	U_{\beta_1}U_{\beta_2}&=U_{\beta_2}U_{\beta_1},(\beta_1,\beta_2 \in \Phi^-, \beta_1+\beta_2 \notin \Phi) \\
	(1+U_{\beta_1})(1+U_{\beta_2})&=(1+U_{(i,k)})^p(1+U_{\beta_2})(1+U_{\beta_1}),(\beta_1=(i,j),\beta_2=(j,k) \in \Phi^-)\\
	(1+U_{\beta_1})(1+U_{\beta_2})&=(1+U_{(k,j)})^{-p}(1+U_{\beta_2})(1+U_{\beta_1}),(\beta_1=(i,j),\beta_2=(k,i) \in \Phi^-) \\
	W_{\delta_1}W_{\delta_2}&=W_{\delta_2}W_{\delta_1},(\delta_1,\delta_2 \in \Pi, \delta_1 \neq \delta_2)\\
	V_{\alpha_1}V_{\alpha_2}&=V_{\alpha_2}V_{\alpha_1},(\alpha_1,\alpha_2 \in \Phi^+, \alpha_1+\alpha_2 \notin \Phi)\\
	(1+V_{\alpha_1})(1+V_{\alpha_2})&=(1+V_{(i,k)})(1+V_{\alpha_2})(1+V_{\alpha_1}),(\alpha_1=(i,j),\alpha_2=(j,k) \in \Phi^+)\\
	(1+V_{\alpha_2})(1+V_{\alpha_1})&=(1+V_{(k,j)})(1+V_{\alpha_1})(1+V_{\alpha_2}),(\alpha_1=(i,j),\alpha_2=(k,i) \in \Phi^+).
	\end{align*}
}
\textbf{\textit{Proof.}} Recalling the notations $h_{\delta}(1+p)$ and $x_{\beta}(p)$ in section \ref{sub:ordiwa}, Steinberg \cite{Yale} gives

\begin{center}
	$h_{\delta}(1+p)x_{\beta}(p)h_{\delta}(1+p)^{-1}=x_{\beta}((1+p)^{\langle \beta , \delta \rangle }p)$, $(\beta\in \Phi^-,\delta \in \Pi)$,
\end{center}
where $\langle \beta , \delta \rangle \in \mathbb{Z}$ (cf. p. $30$ of \cite{Yale}).
So the corresponding relation in the Iwasawa algebra is 
\begin{equation}
(1+W_{\delta})(1+U_{\beta})=(1+U_{\beta})^q(1+W_{\delta}),(\beta \in \Phi^-),
\end{equation}
where $q = (1+p)^{\langle \beta, \delta \rangle}$.\\
\hrule
\vspace{.2cm}
We also have 
\begin{center}
	$h_{\delta}(1+p)x_{\alpha}(1)h_{\delta}(1+p)^{-1}=x_{\alpha}((1+p)^{\langle \alpha , \delta \rangle })$,($\alpha \in \Phi^+$).
\end{center}
So the corresponding relation in the Iwasawa algebra is 
\begin{equation}
(1+W_{\delta})(1+V_{\alpha})=(1+V_{\alpha})^{q^{'}}(1+W_{\delta}),(\alpha \in \Phi^+),
\end{equation}
where $q^{'} = (1+p)^{\langle \alpha, \delta \rangle}$.\\
\hrule
\vspace{.2cm}
If $\alpha \neq -\beta$ and $\alpha+\beta \notin \Phi$, then by example (a), p.$24$ of \cite{Yale} we have
\begin{center}
	$x_{\alpha}(1)x_{\beta}(p)=x_{\beta}(p)x_{\alpha}(1),(\alpha \in \Phi^+,\beta \in \Phi^-).$
\end{center}
So in this case, we have the following relation in the Iwasawa algebra:
\begin{equation}
V_{\alpha}U_{\beta}=U_{\beta}V_{\alpha},(\alpha \in \Phi^+,\beta \in \Phi^-).
\end{equation}
\hrule
\vspace{.2cm}
If on the contrary, $\alpha \neq -\beta$ and $\alpha+\beta \in \Phi$, then we have two subcases:\\

\textit{Subcase 1.} $\alpha=(i,j),\beta=(j,k)$
then  we have by direct computation (writing $x_{(i,j)}(1)=1+E_{i,j}$)
\begin{center}
	$x_{\alpha}(1)x_{\beta}(p)=[x_{\alpha}(1),x_{\beta}(p)]x_{\beta}(p)x_{\alpha}(1)$.\\
	$x_{\alpha}(1)x_{\beta}(p)=x_{(i,k)}(p)x_{\beta}(p)x_{\alpha}(1),((i,k)=\alpha+\beta)$.
\end{center}
Now, in the Iwasawa algebra, $x_{(i,k)}(p)$ corresponds to $(1+V_{(i,k)})$ if $i<k$ and to $(1+U_{(i,k)})$ if $i>k$. \\

Thus, we get the following two relations in the Iwasawa algebra:
\begin{equation}
(1+V_{\alpha})(1+U_{\beta})=(1+V_{(i,k)})^p(1+U_{\beta})(1+V_{\alpha}),i<k,(\alpha \in \Phi^+,\beta \in \Phi^-),
\end{equation}
\begin{equation}
(1+V_{\alpha})(1+U_{\beta})=(1+U_{(i,k)})(1+U_{\beta})(1+V_{\alpha}),i>k,(\alpha \in \Phi^+,\beta \in \Phi^-),
\end{equation}
where $(i,k)=\alpha +\beta =(i,j)+(j,k)$. \\

\textit{Subcase 2}. $\alpha=(i,j),\beta=(k,i)$,
then we have 
\begin{center}
	$x_{\alpha}(1)x_{\beta}(p)=[x_{\alpha}(1),x_{\beta}(p)]x_{\beta}(p)x_{\alpha}(1)$,\\
	$x_{\alpha}(1)x_{\beta}(p)=x_{(k,j)}(-p)x_{\beta}(p)x_{\alpha}(1),(\alpha+\beta=(k,j))$.
\end{center}
So, as before, we get the following two relations in the Iwasawa algebra:
\begin{equation}
(1+V_{\alpha})(1+U_{\beta})=(1+V_{(k,j)})^{-p}(1+U_{\beta})(1+V_{\alpha}),k<j,(\alpha \in \Phi^+,\beta \in \Phi^-),
\end{equation}
\begin{equation}
(1+V_{\alpha})(1+U_{\beta})=(1+U_{(k,j)})^{-1}(1+U_{\beta})(1+V_{\alpha}),k>j,(\alpha \in \Phi^+,\beta \in \Phi^-),
\end{equation}
where $(k,j)=\beta+\alpha =(k,i)+(i,j)$.\\
\hrule
\vspace{.2cm}
If we have $\alpha=-\beta$, let $\alpha=(i,j)[i<j]$, then just by computation one can show that 
\begin{center}
	$x_{\alpha}(1)x_{\beta}(p)=x_{\beta}(p(1+p)^{-1})Hx_{\alpha}((1+p)^{-1})$,
\end{center}
where $H$ is the diagonal matrix with $(1+p)$ in the $(i,i)^{th}$ place and $(1+p)^{-1}$ in the $(j,j)^{th}$ place and $1$ in the other diagonal positions.
The above relation, in $SL_2(\mathbb{Z}_p)$, can be realized by the following matrix equation:

\begin{center}
	$		
	\begin{pmatrix} 1 & 1 \\ 0 & 1 \end{pmatrix}
	\begin{pmatrix} 1 & 0 \\ p & 1 \end{pmatrix}
	= 
	\begin{pmatrix} 1 & 0 \\ p(1+p)^{-1} & 1 \end{pmatrix}
	\begin{pmatrix} (1+p) & 0 \\ 0 & (1+p)^{-1} \end{pmatrix}
	\begin{pmatrix} 1 & (1+p)^{-1} \\ 0 & 1 \end{pmatrix}
	$.
\end{center}

So we obtain
\begin{center}
	$x_{\alpha}(1)x_{\beta}(p)=x_{\beta}(p(1+p)^{-1})h_{(i,i+1)}(1+p)\cdots h_{(j-1,j)}(1+p)x_{\alpha}((1+p)^{-1})$. 
\end{center}
Hence, the corresponding relation in the Iwasawa algebra is 
\begin{equation}
(1+V_{\alpha})(1+U_{-\alpha})=(1+U_{-\alpha})^{(1+p)^{-1}}(1+W_{(i,i+1)})\cdots (1+W_{(j-1,j)})(1+V_{\alpha})^{(1+p)^{-1}},(\alpha \in \Phi^+).
\end{equation}
\hrule
\vspace{.2cm}
In the following, we give relations among the variables corresponding to the lower unipotent subgroup.\\

Let $\beta_1,\beta_2 \in \Phi^{-}, \beta_1+\beta_2 \notin \Phi$, then $x_{\beta_1}(p)x_{\beta_2}(p)=x_{\beta_2}(p)x_{\beta_1}(p)$. So the corresponding relation in the Iwasawa algebra is 
\begin{equation}
U_{\beta_1}U_{\beta_2}=U_{\beta_2}U_{\beta_1},(\beta_1,\beta_2 \in \Phi^-).
\end{equation}
\hrule
\vspace{.2cm}
On the other hand, if $\beta_1+\beta_2\in \Phi$, then we have the following two cases, computations of which actually take place in $GL(3)$:\\

\textit{Case 1}. Let $\beta_1=(i,j),\beta_2=(j,k) [i>j>k]$.
Then, 
\begin{center}
	$x_{\beta_1}(p)x_{\beta_2}(p)=x_{(i,k)}(p^2)x_{\beta_2}(p)x_{\beta_1}(p),(\beta_1+\beta_2=(i,k))$.
\end{center}
The corresponding relation in the Iwasawa algebra is 
\begin{equation}
(1+U_{\beta_1})(1+U_{\beta_2})=(1+U_{(i,k)})^p(1+U_{\beta_2})(1+U_{\beta_1}),(\beta_1,\beta_2 \in \Phi^-,(i,k)=\beta_1+\beta_2),
\end{equation}
since $x_{(i,k)}(p)$ corresponds to $(1+U_{(i,k)})$.\\

\textit{Case 2.} Let $\beta_1=(i,j),\beta_2=(k,i),[k>i>j]$. We have
\begin{center}
	$x_{\beta_1}(p)x_{\beta_2}(p)=x_{(k,j)}(-p^2)x_{\beta_2}(p)x_{\beta_1}(p),(\beta_1+\beta_2=(k,j))$.
\end{center}
The corresponding relation in the Iwasawa algebra is 
\begin{equation}
(1+U_{\beta_1})(1+U_{\beta_2})=(1+U_{(k,j)})^{-p}(1+U_{\beta_2})(1+U_{\beta_1}),(\beta_1,\beta_2 \in \Phi^-,(k,j)=\beta_1+\beta_2).
\end{equation}
\hrule
\vspace{.2cm}
As the diagonal elements commute, we have for $\delta_1,\delta_2 \in \Pi, \delta_1 \neq \delta_2$,
\begin{equation}
W_{\delta_1}W_{\delta_2}=W_{\delta_2}W_{\delta_1},(\delta_1,\delta_2 \in \Pi).
\end{equation}
\hrule
\vspace{.2cm}
Now, we give the relations among the variables corresponding to the upper unipotent subgroup of $G$.\\

If $\alpha_1,\alpha_2 \in \Phi^+, \alpha_1+\alpha_2 \notin \Phi$, then $x_{\alpha_1}(1)x_{\alpha_2}(1)=x_{\alpha_2}(1)x_{\alpha_1}(1)$. So the corresponding relation in the Iwasawa algebra is 
\begin{equation}
V_{\alpha_1}V_{\alpha_2}=V_{\alpha_2}V_{\alpha_1},(\alpha_1,\alpha_2 \in \Phi^+).
\end{equation}
\hrule
\vspace{.2cm}
On the other hand, if $\alpha_1+\alpha_2\in \Phi$, then we have the following two subparts:\\

\textit{Subpart 1}. Let $\alpha_1=(i,j),\alpha_2=(j,k) [i<j<k]$.
Then,
\begin{center}
	$x_{\alpha_1}(1)x_{\alpha_2}(1)=x_{(i,k)}(1)x_{\alpha_2}(1)x_{\alpha_1}(1),(\alpha_1+\alpha_2=(i,k))$.
\end{center}
The corresponding relation in the Iwasawa algebra is 
\begin{equation}
(1+V_{\alpha_1})(1+V_{\alpha_2})=(1+V_{(i,k)})(1+V_{\alpha_2})(1+V_{\alpha_1}),(\alpha_1,\alpha_2 \in \Phi^+,\alpha_1+\alpha_2=(i,k)).
\end{equation}
\textit{Subpart 2.} Let $\alpha_1=(i,j),\alpha_2=(k,i),[k<i<j]$. We have the relation
\begin{center}
	$x_{\alpha_1}(1)x_{\alpha_2}(1)=x_{(k,j)}(-1)x_{\alpha_2}(1)x_{\alpha_1}(1),(\alpha_1+\alpha_2=(k,j))$.
\end{center}
The corresponding relation in the Iwasawa algebra is 
\begin{equation*}
(1+V_{\alpha_1})(1+V_{\alpha_2})=(1+V_{(k,j)})^{-1}(1+V_{\alpha_2})(1+V_{\alpha_1}).
\end{equation*}
which is the same as 
\begin{equation}
(1+V_{\alpha_2})(1+V_{\alpha_1})=(1+V_{(k,j)})(1+V_{\alpha_1})(1+V_{\alpha_2}),(\alpha_1,\alpha_2 \in \Phi^+,\alpha_1+\alpha_2=(k,j)).
\end{equation}
\hrule
\vspace{.2cm}
\newpage
\section{Computations for the proof of Proposition \ref{eq:longlemma}}\label{appendixDlong}

In this section we give the computations necessary to prove proposition 
\ref{eq:longlemma}. We quote proposition \ref{eq:longlemma}.\\

\textbf{\textit{Proposition \ref{eq:longlemma}}}	\textit{For} $m\geqslant 0$, \textit{we have} $\dim $ \textit{gr}$^m\overline{\mathcal{B}} \leqslant d_m=\dim_{\mathbb{F}_p} \text{gr}^m\Omega_G $.\\

\textbf{\textit{Proof}} 
We provide the remaining computations for the proof as indicated in the last paragraph of section \ref{sub:longlemmaiwahori}. There we have already explained the strategy of the proof and dealt with relation \ref{eq:firstIwahori}. Consider relation $\ref{eq:secondIwahori}$. The argument of this exactly follows the case of relation \ref{eq:firstIwahori}  already treated, for which we have shown that we can reduce the number of inversions and so we omit it. Relation $\ref{eq:thirdIwahori}$ is also obvious to deal with, hence we consider $\ref{eq:fiveiwahori}$. It  reduces to 
\begin{center}
	$V_{\alpha}U_{\beta}=U_{\beta}V_{\alpha}$ (mod $Fil^{r+s+1}$),$(\alpha \in \Phi^+,\beta \in \Phi^-)$, 
\end{center}
where $r=$ deg $(V_{\alpha})$ and $s= $deg $(U_{\beta})$.\\

In this case,  we have $\alpha=(i,j),\ \beta=(j,k), \  r=j-i,\ s=n-j+k$ and $k-i=$ deg($V_{(i,k)=\alpha+\beta}$). So, we need to show that 
\begin{center}
	$(1+V_{(i,k)})^p\equiv 1 $(mod $Fil^{n-i+k+1}$).
\end{center}
That is, for any natural number $m \geq 2$, we have to show
that 
\begin{center}
	$(k-i)m \leq n-i+k \implies {p\choose m}\equiv 0$ (mod $p$),
\end{center}
i.e. we have to show that 
\begin{center}
	$(k-i)(m-1)\leq n \implies {p\choose m}\equiv 0$ (mod $p$).
\end{center}
But since $k>i$ (see relation $\ref{eq:fiveiwahori}$), we only have to show that 
\begin{center}
	$(m-1)\leq n \implies {p\choose m}\equiv 0$ (mod $p$),
\end{center}
which we checked in $\ref{eq:equationimp}$.\\

So, if $(x_{b},x_{b+1})=(V_{\alpha},U_{\beta})$, then $x^{i^*}\equiv x^fU_{\beta}V_{\alpha}x^e(Fil^{t+1})$, $t=r^{'}+s^{'}+r+s$. \\

Consider relation $\ref{eq:sixIwahori}$. It is 

\begin{center}
	$(1+V_{\alpha})(1+U_{\beta})=(1+U_{(i,k)})(1+U_{\beta})(1+V_{\alpha}),i>k,(\alpha \in \Phi^+,\beta \in \Phi^-,(i,k)=\alpha+\beta)$
\end{center} 
where $\alpha=(i,j),\beta=(j,k),j>i>k$. \\

Now, deg$(V_{(i,j)})=j-i$,  deg$(U_{(j,k)})=n-j+k$, deg$(V_{\alpha}U_{\beta})=n+k-i=$ deg$(U_{(i,k)})$. Therefore, deg$(U_{(i,k)}V_{\alpha})$ and deg$(U_{(i,k)}U_{\beta})$ are greater than deg$(V_{\alpha}U_{\beta})$. This gives 
\begin{center}
	$V_{\alpha}U_{\beta}=U_{(i,k)}+U_{\beta}V_{\alpha} (Fil^{r+s+1})$,
\end{center}
where $r=$deg$(V_{\alpha}),$ $s=$deg$(U_{\beta})$. If 
\begin{center}
	$(x_{b},x_{b+1})=(V_{\alpha},U_{\beta})$,
\end{center}
then after replacing $V_{\alpha}U_{\beta}$ by $U_{\beta}V_{\alpha}+U_{(i,k)}$, we have to show that we reduce  the number of inversions. For any set $S$, let $\vert S\vert$ denote its cardinality. In the following, the notation $\vert{\substack{\zeta < \mu,\\ i_\zeta>i\mu}}\vert$ will denote $\sum\limits_{\substack{\zeta < \mu\\ i_\zeta>i\mu}}1$. Similarly, $2\vert{\substack{\zeta < \mu,\\ i_\zeta>i\mu}}\vert:=\sum\limits_{\substack{\zeta < \mu\\ i_\zeta>i\mu}}2$. The number of inversions in $x^{i^*}$ was originally as follows, where $\zeta,\mu$ denotes indices $\neq$ $b,b+1$ in the product $x^{i^*}$:
\begin{center}
	$inv=\vert{\substack{\zeta < \mu,\\ i_\zeta>i\mu}}\vert+\vert{\substack{\mu>b+1,\\
			i_{\mu}\in [index(U_{\beta}),index(V_{\alpha})[}}\vert+2\vert{\substack{\mu>b+1,\\
			i_{\mu}<index(U_{\beta})}}\vert+\vert{\substack{\zeta <b,\\
			i_{\zeta}\in ]index(U_{\beta}),index(V_{\alpha})]}}\vert+2\vert{\substack{\zeta <b,\\
			i_{\zeta}>index(V_{\alpha})}}\vert \hspace{.5cm}+ 1$.
\end{center}
Here, the cardinality symbols have natural meanings as explained above in the case for $\vert{\substack{\zeta < \mu,\\ i_\zeta>i\mu}}\vert$. For example, the term $2\vert{\substack{\mu>b+1,\\
		i_{\mu}<index(U_{\beta})}}\vert$ is by definition $\sum\limits_{\substack{\mu>b+1,\\
		i_{\mu}<index(U_{\beta})}}2$. The notation $]index(U_{\beta}),index(V_{\alpha})]$ denotes the half open interval and 
by $index(U_{-})$ we mean that if $U_{-}$  corresponds to $x_{c}$ for some $c \in 1,...,d$ (variable in $\mathcal{A}$), then $index(U_{-})=c$. Thus, $index(U_{\beta})<index(V_{\alpha})$. Now, $V_{\alpha}U_{\beta} \rightarrow U_{\beta}V_{\alpha}$ clearly decreases the number of inversions and after changing $V_{\alpha}U_{\beta}$ into $U_{(i,k)}$,
we have 
\begin{center}
	$inv^{'}=\vert{\substack{\zeta < \mu,\\ i_\zeta>i\mu}}\vert+\vert{\substack{\mu >b+1,\\
			i_{\mu}<index(U_{(i,k)})}}\vert+\vert{\substack{\zeta <b,\\
			i_{\zeta}>index(U_{(i,k)})}}\vert$.
\end{center}
As $index(U_{(i,k)})<index(V_{\alpha})$, we have
\begin{align*}
\vert{\substack{\mu >b+1,\\
		i_{\mu}<index(U_{(i,k)})}}\vert\leq \vert{\substack{\mu >b+1,\\
		i_{\mu}<index(V_{\alpha})}}\vert=\vert{\substack{\mu>b+1,\\
		i_{\mu}\in [index(U_{\beta}),index(V_{\alpha})[}}\vert+\vert{\substack{\mu>b+1,\\
		i_{\mu}<index(U_{\beta})}}\vert \\
\leq \vert{\substack{\mu>b+1,\\
		i_{\mu}\in [index(U_{\beta}),index(V_{\alpha})[}}\vert+2\vert{\substack{\mu>b+1,\\
		i_{\mu}<index(U_{\beta})}}\vert.
\end{align*}

Also, $index(U_{\beta})=index(U_{(j,k)})<indexU_{(i,k)}$ because $j>i$ and our chosen order of the Lazard's basis for $G$. Therefore,
\begin{align*}
\vert{\substack{\zeta <b, \\
		i_{\zeta}>index(U_{(i,k)})}}\vert\leq \vert{\substack{\zeta <b, \\
		i_{\zeta}>index(U_{\beta})}}\vert=\vert{\substack{\zeta <b,\\
		i_{\zeta}\in ]index(U_{\beta}),index(V_{\alpha})]}}\vert+\vert{\substack{\zeta <b,\\
		i_{\zeta}>index(V_{\alpha})}}\vert \\
\leq \vert{\substack{\zeta <b,\\
		i_{\zeta}\in ]index(U_{\beta}),index(V_{\alpha})]}}\vert+2\vert{\substack{\zeta <b,\\
		i_{\zeta}>index(V_{\alpha})}}\vert.
\end{align*}

Hence, $inv^{'}<inv$. \\

\hrule
\vspace{.2cm}
Now, we consider the relation $\ref{eq:sevenIwahori}$. With similar argument as in the case while dealing with the relation $\ref{eq:fiveiwahori}$, using the condition $p>n+1$, it will reduce to 
\begin{center}
	$U_{\beta}V_{\alpha}=V_{\alpha}U_{\beta}(Fil^{r+s+1}),(\alpha \in \Phi^+,\beta \in \Phi^-)$,
\end{center}
where $r=$ deg$(U_{\beta})$ and $s=$ deg$(V_{\alpha})$. So, $V_{\alpha}U_{\beta}\rightarrow U_{\beta}V_{\alpha}$ will obviously reduce the number of inversions. \\
\hrule
\vspace{.2cm}
Let us consider relation $\ref{eightIwahori}$ which is

\begin{center}
	$(1+U_{\beta})(1+V_{\alpha})=(1+U_{(k,j)})(1+V_{\alpha})(1+U_{\beta}),k>j,(k,j)=\beta+\alpha,$
\end{center}
where $\alpha=(i,j),\beta=(k,i),i<j<k.$ Expanding, we obtain 
\begin{center}
	$V_{\alpha}U_{\beta}=U_{\beta}V_{\alpha}-U_{(k,j)}-U_{(k,j)}V_{\alpha}-U_{(k,j)}U_{\beta}(Fil^{r+s+1})$,
\end{center}
where $r=$ deg$(U_{\beta})$ and $s=$ deg$(V_{\alpha})$. Now, deg$(V_{\alpha}U_{\beta})=j-i+n-k+i=n+j-k=$ deg$(U_{(k,j)})$. So, deg$(U_{(k,j)}V_{\alpha})$ and deg$(U_{(k,j)}U_{\beta})$ are greater than deg$(V_{\alpha}U_{\beta})$. Thus,
\begin{center}
	$V_{\alpha}U_{\beta}=-U_{(k,j)}+U_{\beta}V_{\alpha}(Fil^{r+s+1})$.
\end{center}
Let $(x_{b},x_{b+1})=(V_{\alpha},U_{\beta})$. The number of inversions in $x^{i^*}$ was originally as follows, where $\zeta,\mu$ denotes indices $\neq$ $b,b+1$ in the product $x^{i^*}$:

\begin{center}
	$inv=\vert{\substack{\zeta < \mu,\\ i_\zeta>i\mu}}\vert+\vert{\substack{\mu>b+1,\\
			i_{\mu}\in [index(U_{\beta}),index(V_{\alpha})[}}\vert+2\vert{\substack{\mu>b+1,\\
			i_{\mu}<index(U_{\beta})}}\vert+\vert{\substack{\zeta <b,\\
			i_{\zeta}\in ]index(U_{\beta}),index(V_{\alpha})]}}\vert+2\vert{\substack{\zeta <b,\\
			i_{\zeta}>index(V_{\alpha})}}\vert \hspace{.5cm}+ 1$.
\end{center}
As $index(U_{\beta})<index(V_{\alpha})$, the transition $V_{\alpha}U_{\beta} \rightarrow U_{\beta}V_{\alpha}$ clearly reduces the number of inversions and by changing $V_{\alpha}U_{\beta}$ into $-U_{(k,j)}$ we get 
\begin{center}
	$inv^{'}=\vert{\substack{\zeta < \mu,\\ i_\zeta>i\mu}}\vert+\vert{\substack{\mu >b+1,\\
			i_{\mu}<index(U_{(k,j)})}}\vert+\vert{\substack{\zeta <b, \\
			i_{\zeta}>index(U_{(k,j)})}}\vert$.
\end{center}
Now, as $index(U_{(k,j)})<index(V_{\alpha})$ we have 
\begin{align*}
\vert{\substack{\mu >b+1,\\
		i_{\mu}<index(U_{(k,j)})}}\vert\leq \vert{\substack{\mu >b+1,\\
		i_{\mu}<index(V_{\alpha})})}\vert=\vert{\substack{\mu>b+1,\\
		i_{\mu}\in [index(U_{\beta}),index(V_{\alpha})[}}\vert+\vert{\substack{\mu>b+1,\\
		i_{\mu}<index(U_{\beta})}}\vert \\
\leq \vert{\substack{\mu>b+1,\\
		i_{\mu}\in [index(U_{\beta}),index(V_{\alpha})[}}\vert+2\vert{\substack{\mu>b+1,\\
		i_{\mu}<index(U_{\beta})}}\vert.
\end{align*}
Also,  $index(U_{\beta})=index(U_{(k,i)})<indexU_{(k,j)}$ because $i<j$ and our chosen order of the Lazard's basis for $G$. Therefore,
\begin{align*}
\vert{\substack{\zeta <b, \\
		i_{\zeta}>index(U_{(k,j)})}}\vert\leq \vert{\substack{\zeta <b, \\
		i_{\zeta}>index(U_{\beta})}}\vert=\vert{\substack{\zeta <b,\\
		i_{\zeta}\in ]index(U_{\beta}),index(V_{\alpha})]}}\vert+\vert{\substack{\zeta <b,\\
		i_{\zeta}>index(V_{\alpha})}}\vert \\
\leq \vert{\substack{\zeta <b,\\
		i_{\zeta}\in ]index(U_{\beta}),index(V_{\alpha})]}}\vert+2\vert{\substack{\zeta <b,\\
		i_{\zeta}>index(V_{\alpha})}}\vert.
\end{align*}

Hence, $inv^{'}<inv$.\\
\hrule
\vspace{.2cm}
Consider relation $\ref{eq:fourIwahori}$. It is 

\begin{center}
	$(1+V_{\alpha})(1+U_{-\alpha})=(1+U_{-\alpha})^{(1+p)^{-1}}(1+W_{(i,i+1)})\cdots (1+W_{(j-1,j)})(1+V_{\alpha})^{(1+p)^{-1}},(\alpha \in \Phi^+),$
\end{center}
where $\alpha =(i,j)[i<j]$. Let $q=(1+p)^{-1}$. Expanding $(1+U_{-\alpha})^q$, we get that 
\begin{center}
	$(1+U_{-\alpha})^q=1+qU_{-\alpha}+\frac{q(q-1)}{2}U_{-\alpha}^2\cdots $
\end{center}
If $r=$deg$(U_{-\alpha})$, then for any positive integer $m \geq 2$ we have $rm \leq n \implies {q\choose m}\equiv 0$ (mod $p$) because $m \leq rm \leq n <p-1$  trivially implies ${q\choose m} \equiv 0$ (mod $p$) for $m \geq 2$.\\

We have $n=$ deg$(V_{\alpha}U_{-\alpha})$. Expanding the relation above and looking modulo $Fil^{n+1}$ we deduce 
\begin{center}
	$V_{\alpha}U_{-\alpha}=W_{(i,i+1)}+\cdots +W_{(j-1,j)}+U_{-\alpha}V_{\alpha}(Fil^{n+1})$,
\end{center}
since all other product terms will be of the form $W_{-}D$ for some nontrivial variable $D$ and hence have degree strictly greater than $n=$deg$(W_{-})$.\\

Let $(x_{b},x_{b+1})=(V_{\alpha},U_{-\alpha})$, then if we replace $V_{\alpha}U_{-\alpha}$ by $U_{-\alpha}V_{\alpha}$ then we reduce the inversions. If we replace $V_{\alpha}U_{-\alpha}$ by $W_{(k,k+1)}$ then we  show that we reduce the number of inversions. \\
The number of inversions $inv$ in $x^{i^*}$ in the beginning was
\begin{center}
	$\vert{\substack{\zeta < \mu,\\ i_\zeta>i\mu}}\vert+\vert{\substack{\mu>b+1,\\
			i_{\mu}\in [index(U_{-\alpha}),index(V_{\alpha})[}}\vert+2\vert{\substack{\mu>b+1,\\
			i_{\mu}<index(U_{-\alpha})}}\vert+\vert{\substack{\zeta <b,\\
			i_{\zeta}\in ]index(U_{-\alpha}),index(V_{\alpha})]}}\vert+2\vert{\substack{\zeta <b,\\
			i_{\zeta}>index(V_{\alpha})}}\vert \hspace{.5cm}+ 1$.
\end{center}
After changing $V_{\alpha}U_{-\alpha}$ into $W_{k,k+1}$, for some $k\in [i,j-1]$, we count the number of inversions:
\begin{center}
	$inv^{'}=\vert{\substack{\zeta < \mu,\\ i_\zeta>i\mu}}\vert+\vert{\substack{\mu >b+1,\\
			i_{\mu}<index(W_{k,k+1})=\frac{n(n-1)}{2}+k}}\vert+\vert{\substack{\zeta <b, \\
			i_{\zeta}>index(W_{k,k+1})=\frac{n(n-1)}{2}+k}}\vert$.
\end{center}
As $index(W_{k,k+1})<index(V_{\alpha})$, we have
\begin{align*}
\vert{\substack{\mu >b+1,\\
		i_{\mu}<index(W_{k,k+1})}}\vert \leq \vert{\substack{\mu >b+1,\\
		i_{\mu}<index(V_{\alpha})}}\vert=\vert{\substack{\mu>b+1,\\
		i_{\mu}\in [index(U_{-\alpha}),index(V_{\alpha})[}}\vert+\vert{\substack{\mu>b+1,\\
		i_{\mu}<index(U_{-\alpha})}}\vert \\ 
\leq \vert{\substack{\mu>b+1,\\
		i_{\mu}\in [index(U_{-\alpha}),index(V_{\alpha})[}}\vert+2\vert{\substack{\mu>b+1,\\
		i_{\mu}<index(U_{-\alpha})}}\vert.
\end{align*}
Also, as $index(U_{-\alpha})<index(W_{k,k+1})$, we have
\begin{align*}
\vert{\substack{\zeta <b, \\
		i_{\zeta}>index(W_{k,k+1})}}\vert \leq \vert{\substack{\zeta <b, \\
		i_{\zeta}>index(U_{-\alpha})}}\vert=\vert{\substack{\zeta <b,\\
		i_{\zeta}\in ]index(U_{-\alpha}),index(V_{\alpha})]}}\vert+\vert{\substack{\zeta <b,\\
		i_{\zeta}>index(V_{\alpha})}}\vert \\
\leq \vert{\substack{\zeta <b,\\
		i_{\zeta}\in ]index(U_{-\alpha}),index(V_{\alpha})]}}\vert+2\vert{\substack{\zeta <b,\\
		i_{\zeta}>index(V_{\alpha})}}\vert.
\end{align*}

Hence, we obtain $inv^{'}<inv$. \\
\hrule
\vspace{.2cm}
Consider relation $\ref{eq:nineIwahori}$ which is
\begin{center}
	$U_{\beta_1}U_{\beta_2}=U_{\beta_2}U_{\beta_1},(\beta_1,\beta_2 \in \Phi^-).$
\end{center}
So $U_{\beta_1},U_{\beta_2}$ commute and we can reduce the number of inversions. Similarly, the relations $\ref{eq:tenIwahori}$ and $\ref{eq:elevenIwahori}$ will reduce to $U_{\beta_1}U_{\beta_2}=U_{\beta_2}U_{\beta_1}(Fil^{r+s+1})$ where $r=$ deg$(U_{\beta_1})$ and $s=$ deg$(U_{\beta_2})$. (For this use the Lazard condition $p>n+1$ and the computation on degrees as we have already done in $\ref{eq:equationimp}$; example: for $\ref{eq:tenIwahori}$ we have for all natural number $m \geq 2 $, $m-1 \leq $ deg$(U_{i,k})(m-1) \leq n \implies {p\choose m}\equiv 0$ (mod $p$)). So, if we start with the wrong order, that is, suppose $index(U_{\beta_1})>index(U_{\beta_2})$, then $U_{\beta_1}U_{\beta_2}\rightarrow U_{\beta_2}U_{\beta_1}$ reduces the number of inversions. \\
\hrule
\vspace{.2cm}
Relation $\ref{eq:twelveIwahori}$ is 
\begin{center}
	$W_{\delta_1}W_{\delta_2}=W_{\delta_2}W_{\delta_1}(Fil^{2n+1})$,($\delta_1,\delta_2 \in \Pi$). 
\end{center}
So, if we start with the wrong order, that is, suppose $index(W_{\delta_1})>index(W_{\delta_2})$, then $W_{\delta_1}W_{\delta_2}\rightarrow W_{\delta_2}W_{\delta_1}$ reduces the number of inversions. Relation $\ref{eq:thirteenIwahori}$  is similar and so we omit it. We need to struggle with relation $\ref{eq:fourteenIwahori}$ and $\ref{eq:fifteenIwahori}$. \\
\hrule
\vspace{.2cm}
First we consider relation $\ref{eq:fourteenIwahori}$. We have 

\begin{center}
	$(1+V_{\alpha_1})(1+V_{\alpha_2})=(1+V_{(i,k)})(1+V_{\alpha_2})(1+V_{\alpha_1}),(\alpha_1,\alpha_2 \in \Phi^+,\alpha_1+\alpha_2=(i,k))$,
\end{center}
where $\alpha_1=(i,j),\alpha_2=(j,k),i<j<k.$ So we have 
\begin{center}
	$V_{\alpha_1}V_{\alpha_2}=V_{(i,k)}+V_{(i,k)}V_{\alpha_1}+V_{\alpha_2}V_{\alpha_1}(Fil^{r+s+1})$,
\end{center}
where $r:=j-i=$ deg$(V_{\alpha_1})$ and $s:=k-j=$ deg$(V_{\alpha_2})$.  So, the degree of $V_{\alpha_1}V_{\alpha_2}$ is $k-i$ which is the same as the degree of $V_{(i,k)}$. Therefore, we have
\begin{center}
	$V_{\alpha_1}V_{\alpha_2}=V_{(i,k)}+V_{\alpha_2}V_{\alpha_1}(Fil^{r+s+1})$.
\end{center} 
We note that $V_{\alpha_1}=V_{(i,j)}$ and $V_{\alpha_2}=V_{(j,k)}$ and $i<j<k$. So, the wrong order is $V_{\alpha_1}V_{\alpha_2}$ and not $V_{\alpha_2}V_{\alpha_1}$, i.e. $index(V_{\alpha_1})>index(V_{\alpha_2})$. If $(x_{b},x_{b+1})=(V_{\alpha_1},V_{\alpha_2})$, the number of inversions in $x^{i^*}$ was originally as follows, where $\zeta,\mu$ denotes indices $\neq$ $b,b+1$ in the product $x^{i^*}$:

\begin{center}
	$inv=\vert{\substack{\zeta < \mu,\\ i_\zeta>i\mu}}\vert+\vert{\substack{\mu>b+1,\\
			i_{\mu}\in [index(V_{\alpha_2}),index(V_{\alpha_1})[}}\vert+2\vert{\substack{\mu>b+1,\\
			i_{\mu}<index(V_{\alpha_2})}}\vert+\vert{\substack{\zeta <b,\\
			i_{\zeta}\in ]index(V_{\alpha_2}),index(V_{\alpha_1})]}}\vert+2\vert{\substack{\zeta <b,\\
			i_{\zeta}>index(V_{\alpha_1})}}\vert \hspace{.5cm}+ 1$.
\end{center}
The map $ V_{\alpha_1}V_{\alpha_2} \rightarrow  V_{\alpha_2}V_{\alpha_1}$ obviously reduces the number of inversions and by changing \\
$(x_{b},x_{b+1})\rightarrow V_{(i,k)}$ we have 

\begin{center}
	$inv^{'}=\vert{\substack{\zeta < \mu,\\ i_\zeta>i\mu}}\vert+\vert{\substack{\mu >b+1,\\
			i_{\mu}<index(V_{(i,k)})}}\vert+\vert{\substack{\zeta <b, \\
			i_{\zeta}>index(V_{(i,k)})}}\vert$.
\end{center}
Now, $index (V_{(i,k)})<index(V_{(i,j)})=index(V_{\alpha_1})$  and $index(V_{\alpha_2})=index(V_{(j,k)})<index(V_{(i,k)})$ as $k>j>i$,
because of our lexicographic choice of the ordering of the  Lazard's basis for $N^+$. So we have 
\begin{align*}
\vert{\substack{\mu >b+1,\\
		i_{\mu}<index(V_{(i,k)})}}\vert\leq \vert{\substack{\mu >b+1,\\
		i_{\mu}<index(V_{\alpha_1})}}\vert=\vert{\substack{\mu>b+1,\\
		i_{\mu}\in [index(V_{\alpha_2}),index(V_{\alpha_1})[}}\vert+\vert{\substack{\mu>b+1,\\
		i_{\mu}<index(V_{\alpha_2})}}\vert \\
\leq \vert{\substack{\mu>b+1,\\
		i_{\mu}\in [index(V_{\alpha_2}),index(V_{\alpha_1})[}}\vert+2\vert{\substack{\mu>b+1,\\
		i_{\mu}<index(V_{\alpha_2})}}\vert,
\end{align*}
and 
\begin{align*}
\vert{\substack{\zeta <b, \\
		i_{\zeta}>index(V_{(i,k)})}}\vert \leq \vert{\substack{\zeta <b, \\
		i_{\zeta}>index(V_{\alpha_2})}}\vert=\vert{\substack{\zeta <b,\\
		i_{\zeta}\in ]index(V_{\alpha_2}),index(V_{\alpha_1})]}}\vert+\vert{\substack{\zeta <b,\\
		i_{\zeta}>index(V_{\alpha_1})}}\vert \\
\leq \vert{\substack{\zeta <b,\\
		i_{\zeta}\in ]index(V_{\alpha_2}),index(V_{\alpha_1})]}}\vert+2\vert{\substack{\zeta <b,\\
		i_{\zeta}>index(V_{\alpha_1})}}\vert.
\end{align*}

Hence, we obtain $inv^{'}<inv$. \\
\hrule
\vspace{.2cm}
Consider  relation $\ref{eq:fifteenIwahori}$. It is
\begin{center}
	$(1+V_{\alpha_2})(1+V_{\alpha_1})=(1+V_{(k,j)})(1+V_{\alpha_1})(1+V_{\alpha_2}),(\alpha_1,\alpha_2 \in \Phi^+,\alpha_1+\alpha_2=(k,j))$,
\end{center}
where $\alpha_2=(k,i),\alpha_1=(i,j),k<i<j$. Like the previous relation it is evident that the wrong order is $V_{\alpha_2}V_{\alpha_1}$ i.e. $index(V_{\alpha_2})>index(V_{\alpha_1})$ and we have
\begin{center}
	$V_{\alpha_2}V_{\alpha_1}=V_{(k,j)}+V_{\alpha_1}V_{\alpha_2}(Fil^{r+s+1})$,
\end{center}
where $r=$ deg$(V_{\alpha_2})$ and $s=$ deg$(V_{\alpha_1})$. Let $(x_{b},x_{b+1})=(V_{\alpha_2},V_{\alpha_1})$. We count the number of inversions like we did in our previous relation.

\begin{center}
	$inv=\vert{\substack{\zeta < \mu,\\ i_\zeta>i\mu}}\vert+\vert{\substack{\mu>b+1,\\
			i_{\mu}\in [index(V_{\alpha_1}),index(V_{\alpha_2})[}}\vert+2\vert{\substack{\mu>b+1,\\
			i_{\mu}<index(V_{\alpha_1})}}\vert+\vert{\substack{\zeta <b,\\
			i_{\zeta}\in ]index(V_{\alpha_1}),index(V_{\alpha_2})]}}\vert+2\vert{\substack{\zeta <b,\\
			i_{\zeta}>index(V_{\alpha_2})}}\vert \hspace{.5cm}+ 1$.
\end{center}

The map $V_{\alpha_2}V_{\alpha_1} \rightarrow V_{\alpha_1}V_{\alpha_2}$ reduces the number of inversions as $index(V_{\alpha_2})>index(V_{\alpha_1})$ and after changing $(x_{b},x_{b+1})\rightarrow V_{(k,j)}$, we have 

\begin{center}
	$inv^{'}=\vert{\substack{\zeta < \mu,\\ i_\zeta>i\mu}}\vert+\vert{\substack{\mu >b+1,\\
			i_{\mu}<index(V_{(k,j)})}}\vert+\vert{\substack{\zeta <b, \\
			i_{\zeta}>index(V_{(k,j)})}}\vert$.
\end{center}
Here,  $index(V_{(k,j)})<index(V_{(k,i)})=index(V_{\alpha_2})$ and $index(V_{\alpha_1})=index(V_{(i,j)})<index(V_{(k,j)})$ because $k<i<j$ and because of our lexicographic choice of the ordering of the Lazard's basis of $N^+$. Therefore,
\begin{align*}
\vert{\substack{\mu >b+1,\\
		i_{\mu}<index(V_{(k,j)})}}\vert\leq \vert{\substack{\mu >b+1,\\
		i_{\mu}<index(V_{\alpha_2})}}\vert=\vert{\substack{\mu>b+1,\\
		i_{\mu}\in [index(V_{\alpha_1}),index(V_{\alpha_2})[}}\vert+\vert{\substack{\mu>b+1,\\
		i_{\mu}<index(V_{\alpha_1})}}\vert \\
\leq \vert{\substack{\mu>b+1,\\
		i_{\mu}\in [index(V_{\alpha_1}),index(V_{\alpha_2})[}}\vert+2\vert{\substack{\mu>b+1,\\
		i_{\mu}<index(V_{\alpha_1})}}\vert,
\end{align*}
and 
\begin{align*}
\vert{\substack{\zeta <b, \\
		i_{\zeta}>index(V_{(k,j)})}}\vert \leq \vert{\substack{\zeta <b, \\
		i_{\zeta}>index(V_{\alpha_1})}}\vert=\vert{\substack{\zeta <b,\\
		i_{\zeta}\in ]index(V_{\alpha_1}),index(V_{\alpha_2})]}}\vert+\sum\limits_\vert{\substack{\zeta <b,\\
		i_{\zeta}>index(V_{\alpha_2})}}\vert \\
\leq \vert{\substack{\zeta <b,\\
		i_{\zeta}\in ]index(V_{\alpha_1}),index(V_{\alpha_2})]}}\vert+2\vert{\substack{\zeta <b,\\
		i_{\zeta}>index(V_{\alpha_2})}}\vert.
\end{align*}

Hence, we deduce $inv^{'}<inv$.  So we have completed the proof of our Proposition $\ref{eq:longlemma}$.
\vspace{.2cm}
\hrule
\hrule
\vspace{.1cm}
\hrule
\hrule
\hrule
\vspace{.2cm}

\thispagestyle{empty}
\strut\newpage
\bibliographystyle{alpha}
\bibliography{base_change_Iwahori,p_rational_field,Iwawasa_algebra_and_cornut}

\end{document}